\documentclass[12pt,leqno]{amsart}
\usepackage{amsmath,amssymb,txfonts}

      \newtheorem{theorem}{Theorem}[section]
      \newtheorem{example}[theorem]{Example}
      \newtheorem{definition}[theorem]{Definition}
      
      \newtheorem{corollary}[theorem]{Corollary}
      
      \newtheorem{lemma}[theorem]{Lemma}

      \newcommand{\ct}[1]{\langle {#1}\rangle \lower.3ex\hbox{$_{t}$}}
      \newcommand{\lt}[1]{[ {#1}] \lower.3ex\hbox{$_{t}$}}

\begin{document}

\title[Holomorphic Campanato Spaces on the Unit Ball ]{Holomorphic Campanato Spaces on the Unit Ball }
\author{Jianfei Wang and Jie Xiao}
\address{Department of Mathematics and Physics, Information Engineering, Zhejiang Normal University, Zhejiang 321004, P. R. China}
\curraddr{Department of Mathematics and Statistics, Memorial
University, St. John's, NL A1C 5S7, Canada}
\email{wjfustc@zjnu.cn}
\address{Department of Mathematics \& Statistics, Memorial University, NL A1C 5S7, Canada}
         \email{jxiao@mun.ca}
\thanks{JW was in part supported the National Natural Science Foundation of China (No.11001246, No. 11101139) and China Scholarship Council.}
\thanks{JX was in part supported by NSERC of Canada and URP of Memorial University.}

\subjclass[2010]{30H25, 32A10, 32A36, 46G12, 47B38}

\date{}


\keywords{$\mathcal{HC}^s$, $\mathcal{Q}_p$, modified Carleson measures, gradients, associated Gleason problem, induced Carleson problem}

\begin{abstract}
As outlined below, this paper is devoted to a Carleson-type-measure-based study of the holomorphic Campanato $2$-space on the open unit ball $\mathbb B_n$ of $\mathbb C^n$, comprising all Hardy $2$-functions whose oscillations in non-isotropic metric balls on the compact unit sphere $\mathbb S_n$ are proportional to some power of the radius other than the dimension $n\ge 1$.     
\end{abstract}
\maketitle

\tableofcontents

\section*{Introduction}\label{s0}
\setcounter{equation}{0}

For $n\in\mathbb Z^+=\{1,2,3,...\}$ let 
$$
\mathbb B_n=\Big\{(z_1,...,z_n)\in\mathbb C^n:\ \sum_{j=1}^n|z_j|^2<1\Big\}\ \
\&\ \
\mathbb S_n=\Big\{(z_1,...,z_n)\in\mathbb C^n:\ \sum_{j=1}^n|z_j|^2=1\Big\}
$$
be the open unit ball and the compact unit sphere in the complex $n$-space $\mathbb C^n$ respectively. If $\sigma$ stands for the normalized, rotation invariant measure on $\mathbb S_n$, then $\mathcal{H}^2$ represents the Hardy $2$-space of all holomorphic functions on $\mathbb B_n$ such that
$$
\|f\|_{\mathcal{H}^2}=\sup_{r\in (0,1)}\left(\int_{\mathbb S_n}|f(r\zeta)|^2\,d\sigma(\zeta)\right)^\frac12<\infty.
$$
As is well-known, $f\in\mathcal{H}^2$ ensures that 
$$
\begin{cases}
f(\zeta)=\lim_{r\to 1}f(r\zeta)\ \ \hbox{exists\ a.e.\ for}\ \ \zeta\in\mathbb S_n\\
\&\\ 
\|f\|_{\mathcal{H}^2}=\left(\int_{\mathbb S_n}|f(\zeta)|^2\,d\sigma(\zeta)\right)^\frac12,
\end{cases}
$$
and thus $\mathcal{H}^2$ can be identified with a closed subspace of the Lebesgue $2$-space $L^2(\mathbb S_n,\sigma)$.

For $s\in (-1,\frac{n}{2}]$ let $\mathcal{HC}^{s}$ be the Campanato $2$-space of all holomorphic functions $f\in\mathcal H^2$ obeying
$$
\|f\|_{\mathcal{HC}^{s}}=\|f\|_{\mathcal{H}^2}+\sup_{Q(\zeta,r)}\left({r^{2s-n}\int_{Q(\zeta,r)}|f(\xi)-f_{Q(\zeta,r)}|^2\,d\sigma(\xi)}\right)^\frac12<\infty,
$$
where the supremum is taken over all non-isotropic metric balls on $\mathbb S_n$
$$
Q(\zeta,r)=\{\xi\in \mathbb S_n:\,|1-\langle \zeta, \xi\rangle|<r \}
$$ 
and
$$
f_{Q(\zeta,r)}=\int_{Q(\zeta,r)} f(\xi)\,\frac{d\sigma(\xi)}{\sigma(Q(\zeta,r))}.
$$
It is not hard to see that $\mathcal{HC}^{s}$ becomes a Banach space under the norm $\|\,.\,\|_{\mathcal{HC}^{s}}$ and enjoys the following structure table (see, e.g. \cite{Campanato1, Campanato2, FeffermanS, JohnN, Krantz} and \cite[p. 209-217]{KufnerJF} for the real counterparts which are often used in the theory of elliptic partial differential equations):

\bigskip
\begin{center}
    \begin{tabular}{ | l | p{7cm} |}
    \hline
    Index $s$ & Holomorphic Campanato Space $\mathcal{HC}^{s}$\\ \hline
     $s\in (-1,0)$ & Analytic Lipschitz Space $\mathcal A_{-s}$\\ \hline
     $s=0$ & Analytic John-Nirenberg Space $\mathcal{BMOA}$ \\ \hline
     $s\in (0,{n}/{2})$ & Holomorphic Morrey Space $\mathcal{HM}^{s}$\\ \hline
     $s={n}/{2}$ & Holomorphic Hardy Space $\mathcal H^2$ \\ \hline
\end{tabular}
\end{center}
\bigskip

When $n=1$, some fundamental function/operator-theoretic properties of $\mathcal{HC}^s$ have been discovered in \cite{LiLL, WangX, WuX, Xiao2, XiaoX, XiaoY}. However, when $n>1$, as far as we know, there is only one paper partially touching this holomorphic space - more precisely -  \cite{CascanteFO} has established the corona and multiplication theorems for $\mathcal{HC}^s$ under $s\in (0,n/2)$ extending the cases $s\in (-1,0)$ (cf. \cite{KrL}; $s=0$ (cf. \cite{AnC, OrtegaF}); $s=n/2$ (cf. \cite{Am, An}). In light of complex analysis and geometric measure theory, such a higher-dimensional situation is important, interesting, and worth further investigation. This has actually motivated us to carry out this study via the forthcoming five sections. \S\ref{s2} provides basic notations and descriptions of the modified Carleson measures. In \S\ref{s3} we employ the modified Carleson measures to characterize $\mathcal{HC}^s$ by means of the M\"obius transforms and four types of gradients. \S\ref{s4} interprets $\mathcal{HC}^s$ using the holomorphic Q-class and its atomic decomposition (via a fractional differentiation), and \S\ref{s5} discusses the Gleason problem for $\mathcal{HC}^s$ through the canonical example and growth. Finally, in \S\ref{s6}, we deal with the Carleson problem for $\mathcal{HC}^s$, namely, decide geometrically how $\mathcal{HC}^s$ embeds continuously into a more general non-isotropic tent space, and consequently describe when the Riemann-Stieltjes operator is continuous on $\mathcal{HC}^s$.
 
\section{Preliminaries}\label{s2}
\setcounter{equation}{0}

\subsection{Basic notations}\label{s21} Throughout this paper, we make the following conventions.

 \begin{itemize}
 
 \item For the finite complex plane $\mathbb C$ denote by $\mathbb C^n$ the Hilbert space of all points $z=(z_1,...,z_n)\in \mathbb C\times\cdots\times\mathbb C$ equipped with standard inner product 
 $$
 \langle z, w\rangle =\sum\limits_{i=1}^{n}z_i\bar{w_i}\quad\forall\quad
 z=(z_1,...,z_n)\ \ \&\ \ w=(w_1,...,w_n)\in \mathbb C^n
 $$ 
 and the associated norm $|z|={\langle z, z\rangle}^\frac{1}{2}$.
 
 \item Recall 
 $$
 \mathbb B_n=\{z\in\mathbb C^n:\,|z|<1\}\ \ \&\ \ \mathbb S_n=\{z\in\mathbb C^n:\,|z|=1\}.
 $$
Denote by $\nu$ and $\sigma$ the volume measure on $\mathbb B_n$ normalized so that $\nu(\mathbb B_n)=1$ and the surface-area measure on $\mathbb S_n$ normalized so that $\sigma(\mathbb S_n)=1$, respectively. 

\item
For each $a\in \mathbb B_n$, suppose that $\varphi_{a}(z)$ is the
M\"obius automorphism of $\mathbb B_n$ satisfying
 $$
 \varphi_{a}(0)=a,\,\varphi_{a}(a)=0\ \ \&\ \ \varphi_{a}(\varphi_{a}(z))=z.
 $$  
 If $\mathsf{Aut}(\mathbb B_n)$ stands for the group of all biholomorphic automorphisms of $\mathbb B_n$, then an application of \cite[Theorem 2.2.2]{Rudin} derives
\begin{align*}
 1-|\varphi_{a}(z)|^2=\frac{(1-|a|^2)(1-|z|^2)}{|1-\langle z, a\rangle|^2}.
 \end{align*} 

\item Let $\mathcal{H}(\mathbb B_n)$ be the set of all holomorphic functions on $\mathbb B_n.$ For $f\in\mathcal{H}(\mathbb B_n)$, set 
$$
\nabla{f}=(\frac{\partial f}{\partial z_1},...,\frac{\partial f}{\partial z_n})\ \ \&\ \ \mathsf{R}f(z)=\sum\limits_{k=1}^n z_k\frac{\partial f}{\partial z_k}(z)
$$
be the complex gradient and the radial derivative of $f$ respectively. Upon write 
$$
d\lambda(z)=\frac{d\nu(z)}{(1-|z|^2)^{n+1}}
$$ 
for the M\"obius invariant measure in $\mathbb B_n$, one has 
 $$\int_{\mathbb B_n}f(z)d\lambda(z)=\int_{\mathbb B_n}f\circ{\psi}(z)d\lambda(z)\quad\forall\quad
 f\in \mathcal {H}({\mathbb B_n})\ \ \&\ \ \psi\in\mathsf{Aut}(\mathbb B_n).
 $$
 
 \item For $f\in \mathcal{H}(\mathbb B_n)$, let $\tilde\nabla{f}(z)=\nabla(f\circ\varphi_{z}(0))$ be the invariant gradient of $f$. Then  for $f\in H(\mathbb B_n)$, using a direct computation we have by \cite[p. 49]{Zhu1})
 $$
 |\tilde\nabla{f}(z)|^{2}=(1-|z|^2)(|\nabla f|^2-|\mathsf{R} f(z)|^2)=\frac{1}{4}\tilde{\Delta}|f|^2(z).
 $$ 
 and
 $$
(1-|z|^2)|\mathsf{R} f(z)|\leq(1-|z|^2)|\nabla f|  \leq|\tilde\nabla{f}(z)|.
 $$
 
 \item For the Kronecker delta $\delta_{i,j}$, let 
  $$
  \tilde{\Delta}u(z)=4(1-|z|^2)\sum\limits_{i,j=1}^{n}(\delta_{i,j}-\bar{z_i}z_j)\frac{\partial^{2}u(z)}{\partial z_i\partial \bar{z}_j}
  $$
  be the invariant Laplace operator of $u$. Associated with the fundamental solution of this operator is the Green function below:
  $$
  G(z)=\frac{n+1}{2n}\int_{|z|}^{1}{(1-t^2)^{n-1}}{t^{1-2n}}\,dt.
  $$ 
 Clearly,
 $$
 G(z)\approx (1-|z|^2)^n\quad\hbox{as}\quad |z|\to 1^{-}.
 $$
 Upon putting
 $$
 \begin{cases}
 G(z,a)=G(\varphi_a(z))\quad \forall\quad z,a\in\mathbb B_n;\\
 P(z,\omega)=\frac{(1-|z|^2)^n}{|1-\langle z, \omega\rangle|^{2n}}\quad\forall\quad z,\,\omega\in\mathbb B_n,
\end{cases}
 $$
 one has that if $f\in\mathcal{H}^2$, then it follows from \cite[Lemma 2.10]{OrtegaF} or \cite[Theorem 4.23]{Zhu1} that
 \begin{align*}
 \int_{\mathbb S_n}|f(\zeta)-f(0)|^2d\sigma(\zeta)\approx\int_{\mathbb B_n}|\tilde{\nabla}f(z)|^2(1-|z|^2)^nd\lambda(z),
 \end{align*}
and consequently for $\varphi_{a}\in\mathsf{Aut}(\mathbb B_n)$,
 \begin{align*}
 \int_{\mathbb S_n}|f\circ\varphi_{a}(\zeta)-f(a)|^2d\sigma(\zeta)\approx\int_{\mathbb B_n}|\tilde{\nabla}f(z)|^2(1-|\varphi_{a}(z)|^2)^n\,\lambda(z).
 \end{align*}

\item  For a multi-index $m=(m_1,...,m_n)$ and $f\in\mathcal{H}(\mathbb B_n)$, set
 $$
 \begin{cases}
 |m|=m_1+\cdots m_n;\\
 z^m=z_1^{m_1}\cdot\cdot\cdot z_n^{m_n};\\
 \partial^{m}{f}=\frac{\partial ^{|m|}f}{\partial z_{1}^{m_1}\cdot\cdot\cdot\partial z_{n}^{m_n}}.
 \end{cases}
$$
\item In the above and below, ${\mathsf X}\lesssim{\mathsf Y}$ or
${\mathsf X}\gtrsim{\mathsf Y}$ represents ${\mathsf
X}\le \kappa{\mathsf Y}$ or ${\mathsf X}\ge \kappa{\mathsf Y}$ for some constant $\kappa>0$. Similarly, the notation ${\mathsf X}\approx{\mathsf Y}$ means both ${\mathsf X}\lesssim{\mathsf Y}$ and
${\mathsf X}\gtrsim{\mathsf Y}$ hold.
\end{itemize}

\subsection{Modified Carleson measures}\label{s22}

\begin{definition}
\label{l20}
 For $p>0$ and $\zeta\in\mathbb {S}_n$,  we say that a positive Borel measure $\mu$ on $\mathbb {B}_n$ is a $p$-Carleson measure, denoted $\mu\in\mathcal{CM}_p$, if  
$$ 
\|\mu\|_{\mathcal{CM}_p}=\sup_{Q_r(\zeta)}  \left(\frac{\mu({Q}_{r}{(\zeta))}}{r^{np}}\right)^\frac12<\infty,
$$
where the supremun is taken over all Carleson tubes
$${Q}_{r}(\zeta)=\{z\in \mathbb B_n:\,|1-\langle z, \zeta\rangle|<r \}.$$
\end{definition}

In particular, $1$-Carleson measure is the classical Carleson measure. Moreover, it is easy to see that every $p$-Carleson measure must be finite, and a finite positive measure $\mu$ is a $p$-Carleson measure if and only if 
$$ 
\sup \left\{ \frac{\mu({Q}_{r}{(\zeta))}}{r^{np}}:\zeta\in \mathbb S_n, 0<r<\delta\right\}<\infty
$$
holds for any given positive constant $\delta\le 1$.

\begin{lemma}
\label{l21}
Let $p,q\in (0,\infty)$ and $\mu$ be a positive Borel measure on $\mathbb B_n$. Then $\mu\in {\mathcal{CM}}_p$ if and only if 
$$\|\mu\|_{{\mathcal{CM}}_{p,q}}=\sup\limits_{z\in\mathbb B_n}\left(\int_{\mathbb B_n}
\frac{(1-|z|^2)^{nq}}{|1-\langle z, \omega\rangle|^{n(p+q)}}\,d\mu(\omega)\right)^\frac12< \infty.$$
\end{lemma}
\begin{proof} First suppose $\|\mu\|_{{\mathcal{CM}}_{p,q}}<\infty$. For any $\zeta\in
\mathbb S_n$ and $0<r<1$ we take the point $z=(1-r)\zeta$. Note that 
$$1-\langle z, \omega \rangle=(1-r)(1-\langle \zeta, \omega\rangle)+r,$$
so
$$|1- \langle z, \omega \rangle |\leq (1-r)r+r<2r\quad\forall\quad \omega\in
{Q}_{r}(\zeta).
$$ 
This implies  
$$
\label{122}
\frac{(1-|z|^2)^{nq}}{|1-\langle z, \omega\rangle|^{n(p+q)}}
\geq \frac{r^{nq}}{(2r)^{n(p+q)}}\gtrsim \frac{1}{r^{np}}
\quad\forall\quad\omega\in {Q}_{r}(\zeta),
$$
and then
$$ \|\mu\|^2_{{CM}_{p,q}} \geq \int_{ Q_{r}(\zeta)}
\frac{(1-|z|^2)^{nq}}{|1-\langle z, \omega\rangle|^{n(p+q)}}d\mu(\omega)
\gtrsim \frac{\mu({Q}_{r}(\zeta))}{r^{np}}.
$$
Consequently, $\mu\in {\mathcal{CM}}_p$.

Next, assume $\mu \in {\mathcal{CM}}_p$. Since $\mu$ is finite, we have 
$$\sup\limits_{|z|\leq\frac{3}{4}}\int_{\mathbb B_n}
\frac{(1-|z|^2)^{nq}}{|1-\langle z, \omega\rangle|^{n(p+q)}}d\mu(\omega)
\lesssim \mu(\mathbb B_n)\lesssim \|\mu\|^2_{{\mathcal{CM}}_p}<\infty.
$$
If $\frac{3}{4}<|z|<1$, then choosing 
$$
\begin{cases}
k\in\mathbb Z^+;\\
\zeta={z}/{|z|};\\
r_k=2^{k+1}(1-|z|);\\
E_0={Q}_{r_0}(\zeta);\\
E_k={Q}_{r_{k}}(\zeta)\setminus {Q}_{r_{k-1}},
\end{cases}
$$
and using $\mu\in\mathcal{CM}_p$ we obtain 
$$
\mu(E_k)\leq\mu({Q}_{r_{k}}(\zeta))\lesssim 
 2^{n(k+1)p}(1-|z|)^{np}\|\mu\|^2_{\mathcal{CM}_{p}}.
 $$
and
$$
|1-\langle z, \omega\rangle|\geq |z\|1-\langle \zeta, \omega\rangle|-(1-|z|)\geq 2^{k-1}(1-|z|).
\quad\forall\quad \omega\in E_k,
$$
whence
$$
\int_{E_k}\frac{(1-|z|^2)^{nq}}{|1-\langle z, \omega\rangle|^{n(p+q)}}\,d\mu(\omega)\lesssim
\frac{(1-|z|)^{nq}}{\big(2^{k-1}(1-|z|)\big)^{n(p+q)}}\big(2^{n(k+1)p}(1-|z|)^{np}\big)\|\mu\|^2_{\mathcal{CM}_{p}}
\lesssim \frac{\|\mu\|^2_{\mathcal{CM}_{p}}}{2^{nqk}}.
$$
This last estimate gives
$$
\|\mu\|^2_{{CM}_{p,q}}=\sup\limits_{z\in\mathbb B_n}\left(\int_{E_0}+\sum\limits_{k=1}^{\infty}\int_{E_k}\right)\frac{(1-|\omega|^2)^{nq}}{|1-\langle z, \omega\rangle|^{n(p+q)}}\,d\mu(\omega)
\lesssim \|\mu\|^2_{\mathcal{CM}_{p}}\sum\limits_{k=0}^{\infty}2^{-nqk}<\infty.
$$
\end{proof}

\begin{lemma}
\label{l22} \cite{OrtegaF}
Let $s>-1$, $r,\,t\geq 0$ and $r+t-s>n+1$. Then we have 
$$
\int_{\mathbb B_n}\frac{(1-|\zeta|^2)^s}{|1-z\bar{\zeta}'|^r|1-\omega\bar{\zeta}'|^t}dv(\zeta)
\lesssim \begin{cases}{|1-\omega\bar{z}'|^{s+n+1-r-t}}&\mbox{if}\,\, r-s,\,t-s<n+1;
\\\frac{(1-|z|^2)^{s+n+1-r}}{|1-\omega\bar{z}'|^{t}}&\mbox{if}\,\, t-s<n+1<r-s;\\ \frac{(1-|z|^2)^{s+n+1-r}}{|1-\omega\bar{z}'|^{t}}+\frac{(1-|\omega|^2)^{s+n+1-t}}{|1-\omega\bar{z}'|^{r}}&\mbox{if}\,\, r-s,\,t-s>n+1.
\end{cases}
$$
\end{lemma}

Some essential descriptions of $\mathcal{HC}^s$ presented later, depend heavily on the following general invariance for the modified Carleson measures.

\begin{theorem}
\label{l23}
Suppose 
$$
\begin{cases}
p,\eta\in (0, \frac{n+1}{n});\\ 
a>\max\big\{-\frac{1+\eta}{2},-\frac{1+\eta+n(1-p)}{2}\big\};\\
b>\frac{1+\eta}{2}.
\end{cases}
$$ 
For any Lebesgue measurable function $f$ on $\mathbb{B}_n$ let
$$
\mathsf{T}_{a,\,b}f(z)=\int_{\mathbb{B}_n}\frac{(1-|\omega|^{2})^{b-1}}{|1-\langle z, \omega\rangle|^{n+a+b}}f(\omega)dv(\omega)\ \ \forall\ \  z\in\mathbb{B}_n.
$$
If $|f(z)|^{2}(1-|z|^{2})^{\eta}dv(z)$ belongs to $\mathcal{CM}_p$, then $|\mathsf{T}_{a,b}f(z)|^{2}(1-|z|^{2})^{2a+\eta}dv(z)$ also belongs to $\mathcal{CM}_p$.
\end{theorem}
\begin{proof} Set 
$$
d\mu_{f,\eta}(z)=|f(z)|^2(1-|z|^2)^{\eta}dv(z)\ \ \&\ \ d\mu_{\mathsf T_{a,b}f,\eta}(z)=|\mathsf{T}_{a,\,b}f(z)|^2(1-|z|^2)^{2a+\eta}dv(z).
$$
It is enough to prove
$$
\|\mu_{f,\eta}\|_{\mathcal{CM}_p}<\infty\Rightarrow \|\mu_{\mathsf T_{a,b}f,\eta}\|_{\mathcal{CM}_p}<\infty.
$$ 
To do so, let $\delta>0$, $\zeta\in\mathbb S_n$ and $k\in\mathbb Z^+$ such that $2^{k}\delta\leq1$. Then, for the Carleson tube
$$
{Q}_{\delta}(\zeta)=\{z\in \mathbb B_n:\,|1-\langle z, \zeta\rangle|<\delta \},
$$
we have
\begin{align*}
&\mu_{\mathsf T_{a,b}f,\eta}({Q}_{\delta}(\zeta))=
\int_{{Q}_{\delta}(\zeta)}|\mathsf{T}_{a,b}f(z)|^{2}(1-|z|^{2})^{2a+\eta}\,d\nu(z)\\&=
\int_{{Q}_{\delta}(\zeta)}\left( (\int_{{Q}_{2\delta}(\zeta)}+
\int_{\mathbb B_n\setminus {Q}_{2\delta}(\zeta)})\frac{|f(\omega)|(1-|\omega|^{2})^{b-1}}{|1-\langle z, \omega\rangle|^{n+a+b}}d\nu(\omega)
 \right)^{2}{(1-|z|^2)^{2a+\eta}}\,d\nu(z)\\
 &\lesssim
\int_{{Q}_{\delta}(\zeta)}\left( \int_{{Q}_{2\delta}(\zeta)}
\frac{|f(\omega)|(1-|\omega|^{2})^{b-1}}{|1-\langle z, \omega\rangle|^{n+a+b}}d\nu(\omega)\right)^2{(1-|z|^2)^{2a+\eta}}\,d\nu(z)\\
&+
\int_{{Q}_{\delta}(\zeta)}\left( \int_{\mathbb B_n\setminus {Q}_{2\delta}(\zeta)}
\frac{|f(\omega)|(1-|\omega|^{2})^{b-1}}{|1-\langle z, \omega\rangle|^{n+a+b}}d\nu(\omega)\right)^2{(1-|z|^2)^{2a+\eta}}\,d\nu(z)\\&
\equiv\mathsf{Int_1}+\mathsf{Int_2}.
\end{align*}

For $\mathsf{Int_1}$, let
$$
k(z,\omega)=\frac{(1-|z|^2)^{a+\frac{\eta}{2}}(1-|\omega|^2)^{b-1-\frac{\eta}{2}}}{|1-\langle z, \omega\rangle|^{n+a+b}}
$$
and $\mathsf{T}$ be its induced integral operator:
$$
\mathsf{T}f(z)=\int_{\mathbb B_n}f(\omega)k(z,\omega)d\nu(\omega).
$$
Then, upon choosing
$$
\max\Big\{-a-\frac{\eta}{2}-1,-b+\frac{\eta}{2}\Big\}<\gamma<\min\Big\{a+\frac{\eta}{2},b-1-\frac{\eta}{2}\Big\}
$$
and employing Lemma \ref{l22}, we get
$$
\begin{cases}
\int_{\mathbb B_n}k(z,\omega)(1-|\omega|^2)^{\gamma}d\nu(\omega)\lesssim (1-|z|^2)^{\gamma};\\
\int_{\mathbb B_n}k(z,\omega)(1-|z|^2)^{\gamma}d\nu(z)\lesssim (1-|\omega|^2)^{\gamma}.
\end{cases}
$$
Therefore, we apply the Schur's test \cite[Theorem 2.9]{Zhu1} to get that $\mathsf{T}$ is bounded on $\mathcal{L}^2(\mathbb B_n)$. Choosing 
$$
g(\omega)=(1-|\omega|^2)^{\frac{\eta}{2}}f(\omega)1_{Q_{2\delta}}(\zeta)(\omega)
$$
where $1_{Q_{2\delta}}(\zeta)(\omega)$ is the characteristic function of ${Q_{2\delta}(\zeta)}$, we employ the boundedness of $\mathsf{T}$ and $\mu_{f,\eta}\in \mathcal{CM}_p$ to get 
$$
\mathsf{Int_1}\lesssim \int_{\mathbb B_n}\left(\int_{\mathbb B_n}g(\omega)k(z,\omega)d\nu(\omega)
\right)^{2}d\nu(z)\lesssim\int_{\mathbb B_n}|g(z)|^{2}d\nu(z)\lesssim \delta^{np}\|\mu_{f,\eta}\|_{\mathcal{CM}_p}^2.
$$

For $\mathsf{Int_2}$, let 
$$A_j=\{\omega\in\mathbb B_n:\,2^{j}\delta\leq|1-\langle \omega, \zeta\rangle|<2^{j+1}\delta\}\ \ \forall\ \ j\in\mathbb Z^+.
$$
Note that $z\in A_j$ and $\omega\in Q_{\delta}(\zeta)$ ensure
\begin{align*}
&|1-\langle \omega, z\rangle|^{\frac{1}{2}}\geq |1-\langle z, \zeta\rangle|^{\frac{1}{2}}-|1-\langle \omega, \zeta\rangle|^{\frac{1}{2}}
\geq 
(2^{\frac{j}{2}}-1)\delta^{\frac{1}{2}}\geq \frac{1}{2}(\sqrt{2}-1)2^{\frac{j}{2}}\delta^{\frac{1}{2}},
\end{align*}
where the first triangle inequality is from  \cite[p.66]{Rudin}. 
Note that 
$$\int_{Q_{\delta}(\zeta)}(1-|z|^2)^{2a+\eta}d\nu(z)\lesssim \delta^{2a+\eta+n+1}\quad\&\quad
\mathbb B_n\setminus {Q}_{2\delta}(\zeta)=\cup_{j=1}^{\infty}A_j.
$$ 
So, using $a>-\frac{1+\eta+n(1-p)}{2}$, $p\in(0,\frac{n+1}{n})$ and the  
H\"older inequality, we get 
\begin{align*}
&\mathsf{Int_2}\lesssim \int_{{Q}_{\delta}(\zeta)}\left(\sum\limits_{j=1}^{\infty}\int_{A_j}\frac{|f(\omega)|(1-|\omega|^{2})^{b-1}}{|1-\langle z, \omega\rangle|^{n+a+b}}d\nu(\omega)
\right)^2\frac{d\nu(z)}{(1-|z|^2)^{-2a-\eta}}
\\&\lesssim
\int_{{Q}_{\delta}(\zeta)}\left(\sum\limits_{j=1}^{\infty}(2^{j}\delta)^{-(n+a+b)}\int_{Q_{2^{j+1}\delta}(\zeta)（}\frac{|f(\omega)|}{(1-|\omega|^2)^{1-b}}d\nu(\omega)\right)^2 \frac{d\nu(z)}{(1-|z|^2)^{-2a-\eta}}
\\&\lesssim \delta^{1+\eta-n-2b}\left( \sum\limits_{j=1}^{\infty}2^{-j(n+a+b)} \frac{\left[\int_{Q_{2^{j+1}\delta}(\zeta)}|f(\omega)|^2(1-|\omega|^2)^{\eta}d\nu(\omega)\right]^{\frac{1}{2}}}{\left[
\int_{Q_{2^{j+1}\delta}(\zeta)}(1-|\omega|^2)^{2b-2-\eta}\,d\nu(\omega)\right]^{-\frac{1}{2}}} \right)^{2}
\\&\lesssim
\delta^{1+\eta-n-2b}
\left( \sum\limits_{j=1}^{\infty}2^{-j(n+a+b)} (2^{j}\delta)^{\frac{np}{2}}(2^{j}\delta)^{\frac{2b+n-\eta-1}{2}}\|\mu\|_{\mathcal{CM}_p}
\right)^2
\\&\lesssim
\delta^{np}\left(\sum\limits_{j=1}^{\infty}2^{-j(\frac{1+\eta+n-np}{2}+a)}\right)^{2}\|\mu\|^2_{\mathcal{CM}_p}
\\&\lesssim \delta^{np}\|\mu\|^2_{\mathcal{CM}_p}.
\end{align*}
The foregoing estimates for $\mathsf{Int}_1$ and $\mathsf{Int}_2$ imply that $\mu_{\mathsf T_{a,b}f,\eta}\in \mathcal{CM}_{p}$.
\end{proof}

\section{Fundamental characterizations}\label{s3}
\setcounter{equation}{0}

\subsection{Via M\"obius transforms}\label{s31} Fundamentally, we would first like to know:

\medskip
{\it What is the behavior of a function in $\mathcal{HC}^s$ under the M\"obius self-mappings of $\mathbb B_n$\ ?}
\medskip

The following characterization of $\mathcal{HC}^s$ shows that this space is not M\"obius-invariant unless $s=0$.

\begin{theorem}
\label{t31}
Let $f\in\mathcal{H}^2$ and $ s\in(-\frac{1}{2},\frac{n}{2}]$. Then the following statements are equivalent:

{\rm(i)} $f\in \mathcal{HC}^{s}$.

{\rm(ii)} $\|f\|_{\mathcal{HC}^{s},*}=\sup\limits_{a\in\mathbb B_n}(1-|a|^{2})^{s}\|f\circ\varphi_a-f(a)\|_{\mathcal{H}^2}<\infty.$

{\rm(iii)} $\|f\|_{\mathcal{HC}^{s},\star}=\sup\limits_{a\in\mathbb B_n}\left((1-|a|^2)^{2s}\int_{\mathbb B_n}|\tilde{\nabla}f(z)|^2G(z,a)d\lambda(z)\right)^\frac12<\infty.$ 

\end{theorem} 

\begin{proof} (i)$\Leftrightarrow$(ii) This is handled in accordance with three cases below.

 {\it Case 1}. $s\in(-\frac{1}{2}, 0)$. If $f\in\mathcal{HC}^{s}=\mathcal A_{-s}$, using \cite[Theorem 7.2]{Zhu1} we then see that $$f\in\mathcal{HC}^{s}\Leftrightarrow \||f\||_{\mathcal{HC}^s,*}=\sup\limits_{a\in\mathbb B_n}(1-|z|^2)^{s}|\tilde{\nabla}f(z)|<\infty.$$
According to those two integral formulas linking $f$ and $\tilde{\nabla}f$ listed in the sixth bullet of \S\ref{s21} and Lemma \ref{l22} or {\cite[Proposition 1.4.10]{Rudin}}, we have
\begin{align*}
&\|f\|_{\mathcal{HC}^{s},*}^2=\sup\limits_{a\in\mathbb B_n}(1-|a|^{2})^{2s}\|f\circ\varphi_a-f(a)\|_{\mathcal{H}^2}^{2}\\&\approx \sup\limits_{a\in\mathbb B_n}(1-|a|^2)^{2s}\int_{\mathbb B_n}|\tilde{\nabla}f(z)|^{2}(1-|\varphi_a(z)|^2)^n d\lambda(z)
\\&\lesssim
\||f\||_{\mathcal{HC}^s,*}^2\sup\limits_{a\in\mathbb B_n}\int_{\mathbb B_n}(\frac{1-|a|^2}{1-|z|^2})^{2s}(1-|\varphi_a(z)|^2)^n d\lambda(z)
\\&\lesssim 
\||f\||_{\mathcal{HC}^s,*}^2\sup\limits_{a\in\mathbb B_n}\int_{\mathbb B_n}\frac{(1-|a|^2)^{n+2s}(1-|z|^2)^{(-1-2s)}}{|1-\langle a, z \rangle|^{n+1+(-1-2s)+n+2s}}d\nu(z)\\&\lesssim \||f\||_{\mathcal{HC}^s,*}^2.
\end{align*}

Conversely, if $\|f\|_{{HC}^{s},*}<\infty$, setting $g=f\circ \varphi_a-f(a)$, then it is easy to show  
\begin{align*}
&|\tilde{\nabla}f(a)|=|\nabla g(0)|\\&\leq \left(\int_{\mathbb S_n}
|g(\zeta)-g(0)|^2d\sigma(\zeta)\right)^{\frac{1}{2}}\\&=
\left(\int_{\mathbb S_n}
|f\circ \varphi_a(\zeta)-f(a)|^2d\sigma(\zeta)\right)^{\frac{1}{2}}\\&=
\|f\circ \varphi_a-f(a)\|_{\mathcal{H}^2}.
\end{align*}
Hence
\begin{eqnarray*}
\sup\limits_{a\in\mathbb B_n}(1-|a|^2)^{s}|\tilde{\nabla}f(a)|\leq\sup\limits_{a\in\mathbb B_n}(1-|a|^2)^{s}\|f\circ \varphi_a-f(a)\|_{\mathcal{H}^2}=\|f\|_{{\mathcal{HC}}^{s},*}<\infty
\end{eqnarray*}
whence $f\in\mathcal{A}_{-s}=\mathcal{HC}^{s}.$

{\it Case 2}. $s\in[0, \frac{n}{2})$. Given $a\in\mathbb B_n$. If $|a|\leq {\frac{3}{4}}$, then $\|f\|_{\mathcal{HC}^{s},*}<\infty$ follows by $f\in\mathcal{H}^2$. Hence we need only consider $a\in\mathbb B_n$ with $|a|> {\frac{3}{4}}$. To this end, let
$$
Q_{k}=Q({a}/{|a|},4^k(1-|a|))=\Big\{\xi\in \mathbb S_n:\,|1-\langle \xi, {a}/{|a|}\rangle|<4^k(1-|a|)\Big\}
$$
for $k=0,1,...,N$, where $N$ is the smallest natural number such that $Q_N=\mathbb S_n.$ When $k\geq1$ and $\zeta\in Q_k\setminus Q_{k-1}$, according to
$$
1-\langle a, \zeta\rangle=1-|a|+|a|(1-\langle a/|a|,\zeta),
$$
we obtain
$$
|1-\langle a, \zeta\rangle|\approx 4^k(1-|a|)\ \ \&\ \
\sigma(Q_{k+1})\approx\sigma(Q_k)\approx 4^{nk}(1-|a|)^{n}.
$$
Since
$$f(a)=\int_{\mathbb S_n}f(\zeta)P(a,\zeta)d\sigma(\zeta)$$
and 
$$
\int_{\mathbb S_n}|f(\zeta)-f_{Q_0}|^{2}P(a,\zeta)d\sigma(\zeta)
=\int_{\mathbb S_n}|f(\zeta)-f(a)|^2P(a,\zeta)d\sigma(\zeta)+|f(a)-f_{Q_0}|^2,
$$
it follows that
\begin{align*}
&\|f\circ\varphi_a-f(a)\|_{\mathcal{H}^2}^2=
\int_{\mathbb Sn}|f\circ\varphi_{a}(\zeta)-f(a)|^2d\sigma(\zeta)\\&=
\int_{\mathbb Sn}|f(\zeta)-f(a)|^2P(a,\zeta)d\sigma(\zeta)\\&\leq\int_{\mathbb S_n}|f(\zeta)-f_{Q_0}|^2P(a,\zeta)d\sigma(\zeta)\\&=\int_{\mathbb S_n}|f(\zeta)-f_{Q_0}|^2\frac{(1-|a|^2)^{n}}{|1-\langle a, \zeta\rangle |^{2n}}d\sigma(\zeta)\\&\lesssim
\left(\int_{Q_0}+\sum\limits_{k=0}^{N-1}\int_{Q_{k+1}\setminus Q_k}\right)
 |f(\zeta)-f_{Q_0}|^2\frac{(1-|a|^2)^{n}}{|1-\langle a, \zeta\rangle |^{2n}}d\sigma(\zeta)  \\&\lesssim
(1-|a|^2)^{-n}\left(\int_{Q_0}+\sum\limits_{k=0}^{N-1}\int_{Q_{k+1}\setminus Q_k}4^{-2nk}\right)|f(\zeta)-f_{Q_0}|^2d\sigma(\zeta)\\&\lesssim
\int_{Q_0}|f(\zeta)-f_{Q_0}|^2\,\frac{d\sigma(\zeta)}{\sigma(Q_0)}+\sum\limits_{k=0}^{N-1}{4^{-nk}}\int_{Q_{k+1}}|f(\zeta)-f_{Q_0}|^2\,\frac{d\sigma(\zeta)}{\sigma(Q_{k+1})}.
\end{align*}
From the Cauchy-Schwarz inequality and $f\in\mathcal{HC}^{s}$ it follows that
\begin{align*}
&|f_{Q_{k+1}}-f_{Q_k}|=\frac{1}{\sigma (Q_k)}\left|\int_{Q_k}(f(\zeta)-f_{Q_{k+1}})d\sigma(\zeta)\right|\\&\lesssim\left(\frac{1}{\sigma (Q_k)}\int_{Q_k}|f(\zeta)-f_{Q_{k+1}}|^2d\sigma(\zeta)\right)^{\frac{1}{2}}\\&\lesssim\left(\frac{1}{\sigma (Q_{k+1})}\int_{Q_{k+1}}|f(\zeta)-f_{Q_{k+1}}|^2d\sigma(\zeta)\right)^{\frac{1}{2}}\\&\lesssim2^{-n(k+1)}(1-|a|)^{-\frac{n}{2}}2^{(k+1)(n-2s)}(1-|a|)^{\frac{n-2s}{2}}\|f\|_{\mathcal{HC}^{s}}\\&\lesssim4^{-(k+1)s}(1-|a|)^{-s}\|f\|_{\mathcal{HC}^{s}}\\&\lesssim(1-|a|)^{-s}\|f\|_{\mathcal{HC}^{s}},
\end{align*}
and so that
\begin{eqnarray*}
|f_{Q_{k+1}}-f_{Q_0}|\lesssim |f_{Q_{k+1}}-f_{Q_k}|+\cdots+|f_{Q_1}-f_{Q_0}|\lesssim k(1-|a|)^{-s}\|f\|_{\mathcal{HC}^{s}}.
\end{eqnarray*}
Consequently,
$$
\int_{Q_{k+1}}|f(\zeta)-f_{Q_{0}}|^2\,\frac{d\sigma(\zeta)}{\sigma(Q_{k+1})}\lesssim\int_{Q_{k+1}}|f(\zeta)-f_{Q_{k+1}}|^2\,\frac{d\sigma(\zeta)}{\sigma(Q_{k+1})}+|f_{Q_{k+1}}-f_{Q_0}|^2\lesssim k^2(1-|a|)^{-2s}\|f\|_{\mathcal{HC}^{s}}^{2}.
$$
This implies
$$\|f\|_{\mathcal{HC}^{s},*}=\sup\limits_{a\in\mathbb B_n}(1-|a|^{2})^{s}\|f\circ\varphi_a-f(a)\|_{\mathcal{H}^2}\lesssim\|f\|_{\mathcal{HC}^{s}}<\infty.
$$
On the other hand, if $\|f\|_{\mathcal{HC}^{s},*}<\infty$, then for any given $(\zeta,r)\in\mathbb{S}_n\times(0,1)$, we take 
$$
a=(1-r)\zeta\ \ \&\quad Q(\zeta,r)=\{\xi\in \mathbb S_n:\,|1-\langle \zeta, \xi\rangle|<r \},
$$  
thereby estimating
\begin{align*}
& r^{2s-n}\int_{Q(\zeta,r)}|f(\xi)-f_{Q(\zeta,r)}|^2\,d\sigma(\xi)\\&\lesssim (1-|a|^2)^{2s-n}\int_{Q(\zeta,r)}|f(\xi)-f(a)|^2\,d\sigma(\xi)\\&\lesssim (1-|a|^2)^{2s}\int_{Q(\zeta,r)}|f(\xi)-f(a)|^2\frac{(1-|a|^2)^n}{|1-\langle a, \xi\rangle|^{2n}}\,d\sigma(\xi)\\&\lesssim (1-|a|^2)^{2s}\int_{\mathbb{S}_n}|f(\xi)-f(a)|^2\frac{(1-|a|^2)^n}{|1-\langle a, \xi\rangle|^{2n}}\,d\sigma(\xi)\\&
\approx(1-|a|^2)^{2s}\int_{\mathbb{S}_n}|f\circ\varphi_a(\xi)-f(a)|^2\,d\sigma(\xi)
\\&\lesssim\|f\|^2_{\mathcal{HC}^{s},*}.
\end{align*}
Thus $\|f\|_{\mathcal{HC}^{s}}<\infty.$

{\it Case 3}: $s=\frac{n}{2}$. If 
$$
\sup\limits_{a\in\mathbb B_n}(1-|a|^{2})^{2s}\|f\circ\varphi_a-f(a)\|_{\mathcal{H}^2}^{2}<\infty,
$$ 
then 
$$
\|f\|_{\mathcal{H}^2}\lesssim|f(0)|+\|f-f(0)\|_{\mathcal{H}^2}<\infty.
$$ 
Conversely, if $f\in\mathcal{H}^2$, then an application of 
$$
\|f\circ\varphi_a-f(a)\|^{2}_{\mathcal{H}^2}\\=\int_{\mathbb S_n}|f\circ\varphi_a(\zeta)-f(a)|^{2}d\sigma(\zeta)\approx \int_{\mathbb B_n}|\tilde{\nabla}(z)|^{2}(1-|\varphi_a(z)|^2)^n d\lambda(z)\ \ \forall\ \ a\in\mathbb B_n
$$
derives
\begin{align*}
&\sup_{a\in\mathbb B_n}(1-|a|^2)^{n}\|f\circ\varphi_a-f(a)\|_{\mathcal{H}^2}^2\\
&\approx\sup_{a\in\mathbb B_n}(1-|a|^2)^{n}\int_{\mathbb B_n}|\tilde{\nabla}f(z)|^{2}(1-|\varphi_a(z)|^2)^n d\lambda(z)\\
&\approx\sup_{a\in\mathbb B_n}\int_{\mathbb B_n}|\tilde{\nabla}f(z)|^{2}(1-|z|^2)^n\frac{(1-|a|^2)^{2n}}{|1-\langle a, z\rangle|^{2n}} d\lambda(z)
\\&\lesssim \int_{\mathbb B_n}|\tilde{\nabla}f(z)|^{2}(1-|z|^2)^n d\lambda(z)\\&\approx
\|f-f(0)\|_{\mathcal{H}^2}^2\lesssim |f(0)|^2+\|f\|_{\mathcal{H}^2}^2<\infty.
\end{align*}

(ii)$\Leftrightarrow$(iii) This is a consequence of the above-verified equivalence (i)$\Leftrightarrow$(ii) and the well-known estimate
$$
\int_{\mathbb B_n}|\tilde{\nabla}f(z)|^2G(z,a)d\lambda(z)\approx\int_{\mathbb S_n}|f\circ\varphi_{a}(\zeta)-f(a)|^2d\sigma(\zeta).
$$
\end{proof}

\subsection{Via four different gradients}\label{s32} Based on Theorem \ref{t31} and the modified Carleson measures, the following assertion reveals four gradient behaviors of $\mathcal{HC}^s$.
\begin{theorem}
\label{t33}
Let $f\in\mathcal{H}^2$ and $s\in(-\frac{1}{2},\frac{n}{2})$. Then the following statements are equivalent:

{\rm(i)} $f\in \mathcal{HC}^{s}.$

{\rm(ii)} ${|\tilde{\nabla}f(z)|}{(1-|z|^2)^{-1}}d\nu(z)$ belongs to $\mathcal{CM}_{1-\frac{2s}{n}}$.

{\rm(iii)} $|{\nabla}f(z)|^2(1-|z|^2)d\nu(z)$ belongs to $\mathcal{CM}_{1-\frac{2s}{n}}$.

{\rm(iv)} $|\mathsf{R}f(z)|^2(1-|z|^2)d\nu(z)$ belongs to $\mathcal{CM}_{1-\frac{2s}{n}}$.

{\rm(v)} $\sum_{i<j}|\mathsf{T}_{i,j}f(z)|^2d\nu(z)$ belongs to $\mathcal{CM}_{1-\frac{2s}{n}}$ where $\mathsf{T}_{i,j}=
\bar{z_j}\frac{\partial f}{\partial z_i}-\bar{z_i}\frac{\partial f}{\partial z_j}.$
\end{theorem}
\begin{proof} The argument is divided into four steps below.

\par{\it Step 1}: Note that
\begin{align*}
&(1-|a|^2)^{2s}\|f\circ\varphi_a-f(a)\|_{\mathcal{H}^2}^{2}\\
&\approx(1-|a|^2)^{2s}\int_{\mathbb B_n}|\tilde{\nabla}f(z)|^{2}(1-|\phi_a(z)|^2)^n d\lambda(z)\\&\approx(1-|a|^2)^{2s} \int_{\mathbb B_n}\Big(\frac{|\tilde{\nabla}f(z)|^{2}}{(1-|z|^2)^{n+1}}\Big)\Big(\frac{(1-|a|^2)^{n}(1-|z|^2)^n}{|1-\langle a, z\rangle|^{2n}}\Big)\, d\nu(z)
\\&\approx\int_{\mathbb B_n}\Big(\frac{|\tilde{\nabla}f(z)|^{2}}{1-|z|^2}\Big)\Big(\frac{(1-|a|^2)^{n+2s}}{|1-\langle a, z\rangle|^{2n}}\Big)\, d\nu(z)
\\&\approx\int_{\mathbb B_n}\Big(\frac{|\tilde{\nabla}f(z)|^{2}}{1-|z|^2}\Big)\Big(\frac{(1-|a|^2)^{n(1+\frac{2s}{n})}}{|1-\langle a, z\rangle|^{n((1+\frac{2s}{n})+(1-\frac{2s}{n}))}}\Big)\, d\nu(z).
 \end{align*}
An application of Theorem \ref{t31} implies that (i) and (ii) are equivalent.\\

{\it Step 2}: It is easy to see that both (ii)$\Rightarrow$(iii) and (iii)$\Rightarrow$ (iv) hold.

{\it Step 3}: Assume (iv) hods. Notice that 
$$f(z)-f(0)=\int_{0}^{1}\frac{d f(tz)}{dt}dt=\int_{0}^{1}\frac{\mathsf{R} f(tz)}{t}dt.$$
Hence 
$$\mathsf{T}_{i,j}f(z)=\int_{0}^{1}\frac{(\mathsf{T}_{i,j}\mathsf{R} f)(tz)}{t}dt.$$
To prove that $\sum_{i<j}|\mathsf{T}_{i,j}f(z)|^2d\nu(z)$ belongs to $\mathcal{CM}_{1-\frac{2s}{n}}$, it is sufficient to prove 
$$
\left(\int_{\frac{1}{2}}^{1}|(\mathsf{T}_{i,j}\mathsf{R} f)(tz)|dt\right)^2d\lambda(z)\in\mathcal{CM}_{1-\frac{2s}{n}}\ \ \forall\ \ 1\leq i,j\leq n.
$$ 
For any $\alpha>0$, according to \cite[p.3374]{Jev2}, we get 
\begin{eqnarray*}
\int_{\frac{1}{2}}^{1}|(\mathsf{T}_{i,j}\mathsf{R} f)(tz)|dt\leq\int_{0}^{1}|(\mathsf{T}_{i,j}\mathsf{R} f)(tz)|dt\lesssim \int_{\mathbb B_n}\frac{(1-|z|^2)^{\alpha}|(\mathsf{R} f)(z)|}{|1-\langle \omega, z\rangle|^{n+\alpha+\frac{1}{2}}}d\nu(z). 
\end{eqnarray*}
Theorem \ref{l23} is used to show that 
$$\left(\int_{\frac{1}{2}}^{1}|(\mathsf{T}_{i,j}\mathsf{R} f)(tz)|dt\right)^2d\lambda(z)$$
belongs to $\mathcal{CM}_{1-\frac{2s}{n}}$ 
by taking $b=1+\alpha$, $\eta=1$ and $a=-\frac{1}{2}.$
This in turns proves that $$\sum_{i<j}|\mathsf{T}_{i,j}f(z)|^2d\nu(z)$$
belongs to $\mathcal{CM}_{1-\frac{2s}{n}}$. Thus (v) holds.

{\it Step 4}: Now, if (v) holds, then applying \cite[Lemma 2]{Hu2} or \cite[Lemma 2.3]{Jev2}, and the equality in \cite[Lemma 2.2]{Jev2}:
$$
|z|^2|\tilde{\nabla}f(z)|^2=(1-|z|^2)\left((1-|z|^2)|\mathsf {R}f(z)|^2+\sum_{i<j}|\mathsf{T}_{i,j}f(z)|^2 \right),
$$
we have that (i) holds.

\end{proof}

\section{M\"obius invariant counterpart}\label{s4}

\setcounter{equation}{0}
\subsection{Relationship to $\mathcal{Q}_{p}$}\label{s41} Observe that Theorem \ref{t31} tells us that $\|\cdot\|_{\mathcal{HC}^s}$ varies under the M\"obius self-mappings of $\mathbb B_n$ unless $s=0$. So, it is very natural to ask the following question:

\medskip
{\it What is the M\"obius invariant counterpart of $\mathcal{HC}^s$ for each $s\not=0$\ ?}

\medskip

Fortunately, answering this question leads to a consideration of the so-called holomorphic $\mathcal{Q}_p$-spaces with $p>0$ (cf. \cite{OuYZ}):

$$
\mathcal{Q}_p=\left\{f\in \mathcal{H}(\mathbb B_n):\,\,\|f\|_{\mathcal{Q}_p}=\sup\limits_{\omega\in\mathbb B_n}\left({\int_{\mathbb B_n}|{\nabla}f(z)|^2 G^{p}(z,\omega)d\nu(z)}\right)^\frac12<\infty\right\},
$$
whose one-dimensional case goes back to \cite{AXZ} (cf. \cite{EssenX, Xiao1, Xiao2}). Clearly, each $\mathcal{Q}_p$ is M\"obius invariant. More importantly, $\mathcal{Q}_p$ enjoys the following structure table:

\bigskip

\begin{center}
    \begin{tabular}{ | l | p{8cm} |}
    \hline
    Index $p$ & Space $\mathcal{Q}_p$  \\ \hline
     $0<p\leq \frac{n-1}{n}$ & $\mathbb C$ \\
      \hline
     $\frac{n-1}{n}<p<1$ & $\mathcal{Q}_p$ with fractional scale \\ \hline
     $p=1$ & Analytic John-Nirenberg space $\mathcal{BMOA}$\ \cite{CRW}\\  \hline
     $1<p<\frac{n}{n-1}$ & Bloch space\ \cite{OuYZ} \\ \hline 
     $p\geq \frac{n}{n-1}$ &  $\mathbb C$ \\ \hline
\end{tabular}
\end{center}
\bigskip

Comparing this table and the structure table of $\mathcal{HC}^s$, and taking \cite{WuX} (for $n=1$) into an account, we ask ourselves whether there is any connection between $\mathcal{HC}^s$ and $\mathcal{Q}_p$ under $n\ge 2$. To treat this question, let us recall the concept of the fractional derivatives.

 Suppose that
 $f\in\mathcal{H}(\mathbb B_n)$ has the homogeneous expansion 
 $$
 f(z)=\sum\limits_{k=0}^{\infty}f_k(z).
 $$ 
 For such real parameters $\alpha$ and $t$ that neither $n+\alpha$ nor $n+\alpha+t$ is a negative integer, the radial fractional derivative is defined below:
\begin{align*}
\mathsf{R}^{\alpha,t} f(z)=\sum_{k=0}^{\infty}\frac{\Gamma(n+1+\alpha)\Gamma(n+1+k+\alpha+t)}{\Gamma(n+1+\alpha+t)\Gamma(n+1+k+\alpha)}f_k(z),
\end{align*}
where $\Gamma(\cdot)$ is the classical gamma function. Then, $\mathsf{R}^{\alpha,t}$ exists as an operator mapping $\mathcal{H}(\mathbb B_n)$ to $\mathcal{H}(\mathbb B_n)$, and its inverse $\mathsf{R}_{\alpha,t}$ is determined by:
$$ 
\mathsf{R}_{\alpha,t} f(z)=\sum_{k=0}^{\infty}\frac{\Gamma(n+1+\alpha+t)\Gamma(n+1+k+\alpha)}{\Gamma(n+1+\alpha)\Gamma(n+1+k+\alpha+t)}f_k(z).
$$ 
In particular, 
$$
\mathsf{R}^{\alpha,0}\equiv\mathsf{I}\quad\&\quad \mathsf{R}^{\alpha,t_1}\circ\mathsf{R}^{\alpha+t_1,t_2}=\mathsf{R}^{\alpha,t_1+t_2}.
$$
An application of Stirling's formula yields that
$$
\frac{\Gamma(n+1+\alpha)\Gamma(n+1+k+\alpha+t)}{\Gamma(n+1+\alpha+t)\Gamma(n+1+k+\alpha)}\approx k^t,
$$
and so that
$$
\mathsf{R}^{\alpha,t} f(z)\quad\&\quad \mathsf{R}_{\alpha,t} f(z)
$$
can be replaced by
$$
 \sum_{k=0}^{\infty}k^tf_k(z)\quad\&\quad
  \sum_{k=0}^{\infty}k^{-t}f_k(z),
  $$
  respectively.

\begin{theorem}
\label{t41}
Let $s\in(-\frac{1}{2},\frac{1}{2})$, $t\in(0,\infty)$ and $f\in\mathcal{H}(\mathbb B_n)$. If neither $n+\alpha$ nor $n+\alpha+t$ is a negative integer, then the following three statements are equivalent:

{\rm(i)} $f\in \mathcal{HC}^s$.

{\rm (ii)} For any $t>\max\{0,-s\}$, the measure $|\mathsf{R}^{\alpha,t} f(z)|^{2}(1-|z|^2)^{2t-1}d\nu(z)$ belongs to $\mathcal{CM}_{1-\frac{2s}{n}}. $

{\rm (iii)} $\mathsf{R}^{\alpha,-s} f \in\mathcal{Q}_{1-\frac{2s}{n}}$.
\end{theorem}

\begin{proof} The argument is divided into the following two steps.

{\it Step 1}: (i)$\Leftrightarrow$(ii). According to 
Theorem \ref{t33}, one has
$$
f\in \mathcal{HC}^s\Leftrightarrow |\mathsf{R} f(z)|^{2}(1-|z|^2)d\nu(z)\in\mathcal{CM}_{1-\frac{2s}{n}}.
$$ 
This means that (ii)$\Rightarrow$(i) is trivial. Conversely, suppose that (i) is valid. Then we use \cite {BeaB} and \cite {Jev1} to get
$$
|\mathsf{R}^{\alpha,t} f(z)|\lesssim
\int_{\mathbb B_n}\frac{|\mathsf{R} f(\omega)|(1-|\omega|^2)^{b-1}}{|1-\langle z, \omega\rangle|^{n+b+(t-1)}}d\nu(\omega)\quad\hbox{as}\quad (t,b)\in (0,1)\times (0,\infty).
$$
In view of Theorem \ref{l23}, it follows that 
$$|\mathsf{R}^{\alpha,t} f(z)|^2(1-|z|^2)^{2t-1}d\nu(z)\in\mathcal{CM}_{1-\frac{2s}{n}}
$$ 
by setting $\eta=1$, $a=t-1$ and $p=1-\frac{2s}{n}$. When $t>1$, a slight modification of the proof of \cite [Theorem 1] {Jev1} derives that 
$$
|\mathsf{R}^{\alpha,t} f(z)|^2(1-|z|^2)^{2t-1}d\nu(z)\in\mathcal{CM}_{1-\frac{2s}{n}}.
$$ 
Thus (ii) holds.

{\it Step 2}: (i)$\Leftrightarrow$(iii). {\it Step 1} indicates 
$$
f\in\mathcal{HC}^s\Leftrightarrow
|\mathsf{R}^{\alpha,t}f(z)|^2(1-|z|^2)^{2t-1}d\nu(z)\in\mathcal{CM}_{1-\frac{2s}{n}}.
$$
Taking $t=1-s$ we obtain 
$$
2t-1=1-2s=1+n[(1-\frac{2s}{n})-1].
$$
By applying \cite[Corollary 3.2]{LiO} to $p=1-\frac{2s}{n}$ and $m=1$, we can show 
$$
f\in\mathcal{HC}^s\Leftrightarrow\mathsf{R}^{\alpha,-s}f\in\mathcal{Q}_{1-\frac{2s}{n}}.
$$ 
Thus (iii)$\Rightarrow$(i) follows.
 
Conversely, suppose that (i) is true. Thanks to (i)$\Leftrightarrow$(ii) we get 
$$
\mathsf{R}^{\alpha,s} f\in\mathcal{HC}^s\Leftrightarrow 
 |\mathsf{R}^{\alpha,t}\circ\mathsf{R}^{\alpha+t,s} f(z)|^2(1-|z|^2)^{2t-1}d\nu(z)
\in\mathcal{CM}_{1-\frac{2s}{n}}.
$$
Note that 
$$
\mathsf{R}^{\alpha,t}\circ\mathsf{R}^{\alpha+t,s}=
\mathsf{R}^{\alpha,t+s}.
$$
So, setting $t=1-s$, along with \cite[Corollary 3.2]{LiO}, we have $f\in\mathcal{Q}_{1-\frac{2s}{n}}$, whence validating (iii).
\end{proof}

Remarkably, $\mathcal{HC}^s$ is isomorphic to $\mathcal {Q}_{1-\frac{2s}{n}}$ via a fractional differentiation:
$$
\mathcal{HC}^s=
\begin{cases}
\Lambda_{-s} \cong \mathcal{B}\,\,\hbox{as}\ \ -\frac{1}{2}<s<0;\\
\mathcal{BMOA}\cong \mathcal{BMOA}\,\,\hbox{as}\ \ s=0;\\
\mathcal{HM}^{s}\cong \mathcal{Q}_{1-\frac{2s}{n}}\,\,\hbox{as}\ \ 0<s<\frac{1}{2},
\end{cases}
$$
where $\cong$ stands for isomorphism. As a simple application of Theorem \ref{t41}, we obtain the structure assertion on $\mathcal{Q}_p$ below.

\begin{corollary}\label{t42} Let $s\in(-\frac{1}{2},\frac{1}{2})$ and $f\in\mathcal{H}(\mathbb B_n)$. Then 
the following statements are equivalent:

{\rm(i)} $f\in\mathcal{Q}_{1-\frac{2s}{n}}$.

{\rm (ii)} For any $t>\max\{0,-s\}$, $|\mathsf{R}^{\alpha,t} f(z)|^{2}(1-|z|^2)^{2t-2s-1}d\nu(z)$ belongs to $\mathcal{CM}_{1-\frac{2s}{n}}$.

{\rm (iii)} $\mathsf{R}^{\alpha,s} f\in\mathcal{HC}^s$.
\end{corollary}
\begin{proof}
The argument is similar to follow from Theorem \ref{t41}.
\end{proof}

\subsection{Atomic decompositon}\label{s42} Atomic decompositions for the Bloch space, $\mathcal{BMOA}$ and $\mathcal{Q}_p$ (with fractional order $p$) have been obtained in \cite{Rochberg}, \cite{RochbergS} and \cite{WuX2, PengOu1}, respectively. This observation plus Corollary \ref{t42} suggests us to find an atomic decomposition for $\mathcal{HC}^s$. To do so, let
$$
\beta(z,\omega)=\frac{1}{2}\log\frac{1+|\varphi_{z}(\omega)|}{1-|\varphi_{z}(\omega)|}
$$ 
be the Bergman distance between $z$ and $\omega$ in $\mathbb B_n$, where 
$$
\varphi_z\in\mathsf{Aut(\mathbb B_n)}\ \ \hbox{with}\ \  \varphi_z(0)=z\ \ \&\ \ \varphi_z(z)=0.
$$
Suppose $r>0$ and $z\in\mathbb B_n$, the set 
$$
\mathbb{B}(z,r)=\{\omega\in\mathbb B_n:
\,\beta(z,\omega)<r\}
$$
is called a Bergman metric ball centered at $z$. Given an $N\in\mathbb Z^+$. A sequence $\{a_k\}\subset \mathbb B_n$ is called an $r$-lattice in the Bergman metric if $\{a_k\}$ has the following properties: 
\begin{itemize}
\item $\mathbb {B}_n=\cup_{k}\mathbb{B}(a_k,r)$.\\

\item The sets $\mathbb{B}(a_k,r/4)$ are mutually disjoint.\\

\item Each point $z\in\mathbb B_n$ belongs to at most $N$ of the sets $\mathbb{B}(a_k,r/4).$ \\

\item $\beta (a_i, a_j)\geq \frac{r}{2}$ for all $i\neq j.$
\end{itemize}

By \cite[Theorem 2.23]{Zhu1}, we can select an $N\in\mathbb Z^+$ such that for any $1<r\leq 1$ there is an $r$-lattice $\{a_k\}\subset\mathbb B_n$. For $a\in\mathbb B_n$ we write $\delta_{a}$ to denote the unit point-mass at the point $a$. 

\begin{lemma}
\label{t44}
Assume $ s\in(-\frac{1}{2},\frac{1}{2})$ and $b>n-s$.

{\rm (i)} Let $\{a_k\}$ be an $r$-lattice as described above. If $\{c_k\}$ is such a complex sequence that
$$
\sum\limits_{k=1}^{\infty}|c_k|^2(1-|a_k|^2)^{n-2s}\delta_{ a_k}
$$
belongs to $\mathcal{CM}_{1-\frac{2s}{n}}$, then 
$$
f(z)=\sum\limits_{k=1}^{\infty}c_k\frac{(1-|a_k|^2)^{b}}{(1-\langle z, a_k\rangle)^b}
$$
is an element of $\mathcal{Q}_{1-\frac{2s}{n}}$.

{\rm (ii)} If $f\in\mathcal{Q}_{1-\frac{2s}{n}}$, then there exists an $r$-lattice $\{a_k\}\subset\mathbb B_n$ (as described above) and a complex sequence $\{c_k\}$ such that 
$$
f(z)=\sum\limits_{k=1}^{\infty}c_k\frac{(1-|a_k|^2)^{b}}{(1-\langle z, a_k\rangle)^b},
$$
and
$$
\sum\limits_{k=1}^{\infty}|c_k|^2(1-|a_k|^2)^{n-2s}\delta_{ a_k}
$$
belongs to $\mathcal{CM}_{1-\frac{2s}{n}}$,
\end{lemma}

\begin{proof} Let 
$$
d\mu=\sum\limits_{k=1}^{\infty}|c_k|^2(1-|a_k|^2)^{n-2s}\delta_{ a_k}.
$$

{\it Case 1}: $0\leq s<\frac{1}{2}$. This ensures $p=1-\frac{2s}{n}\in(\frac{n-1}{n},1]$. So, the decomposition assertion under this case follows from \cite[Theorem 1]{PengOu1}.

{\it Case 2}: $-\frac{1}{2}< s<0$. Under this situation, one has that $\mathcal{Q}_{1-\frac{2s}{n}}$ coincides with the Bloch space $\mathcal{B}$ which, according to \cite[Theorem 3.23]{Zhu1}, consists of the following functions 
$$
f(z)=\sum\limits_{k=1}^{\infty}c_k\frac{(1-|a_k|^2)^{b}}{(1-\langle z, a_k\rangle)^b},
$$
where $\{c_k\}\in l^{\infty}.$ Hence, it suffices to prove that 
$\{c_k\}\in l^{\infty}$ if and only if 
$$
\sum\limits_{k=1}^{\infty}|c_k|^2(1-|a_k|^2)^{n-2s}\delta_{a_k}
$$ 
belongs to $\mathcal{CM}_{1-\frac{2s}{n}}$. In fact, if 
$$
\{c_k\}\in l^{\infty}\ \ \&\ \ d\mu=\sum\limits_{k=1}^{\infty}|c_k|^2(1-|a_k|^2)^{n-2s}\delta_{a_k},
$$ 
then
$$
\int_{\mathbb B_n}\frac{(1-|\omega|^2)^n}{|1-\langle z,\omega\rangle|^{2n-2s}}d\mu(z)
=\sum\limits_{k=1}^{\infty}|c_k|^2\frac{(1-|a_k|^2)^{n-2s}(1-|\omega|^2)^n}{|1-\langle a_k,\omega\rangle|^{2n-2s}}.
$$
Given $\omega\in\mathbb B_n$.  Using \cite[Lemma 2.24]{Zhu1} or the subharmonicity of 
$$
f(z)={(1-\langle z,\, \omega \rangle)^{2s-n}},
$$ 
we have 
$$
{|1-\langle a_k,\omega\rangle|^{2s-2n}}\lesssim {(1-|a_k|^2)^{2s-n}}\int_{\mathbb B(a_k, r)}\frac{(1-|z|^2)^{-1-2s}}{|1-\langle z,\omega\rangle|^{2n-2s}}d\nu(z),
$$
whence getting
$$
\frac{(1-|a_k|^2)^{n-2s}}{|1-\langle a_k,\omega\rangle|^{2n-2s}}\lesssim
\int_{\mathbb {B}(a_k, r/4)}\frac{(1-|z|^2)^{-1-2s}}{|1-\langle z,\omega\rangle|^{2n-2s}}d\nu(z).
$$
Since $\mathbb B_n=\cup_{k} \mathbb {B}(a_k,r/4)$, an application of the last estimate implies 
\begin{align*}
&\sum\limits_{k=1}^{\infty}\frac{(1-|a_k|^2)^{n-2s}}{|1-\langle a_k,\omega\rangle|^{2n-2s}}\\
&\lesssim
\int_{\cup_{k}\mathbb {B}(a_k, r/4)}\frac{(1-|z|^2)^{-1-2s}}{|1-\langle z,\omega\rangle|^{2n-2s}}d\nu(z)\\
&\lesssim
\int_{\mathbb B_n}\frac{(1-|z|^2)^{-1-2s}}{|1-\langle z,\omega\rangle|^{2n-2s}}d\nu(z)\\
&\lesssim {(1-|\omega|^2)^{-n}}
\end{align*}
where \cite[Proposition 1.4.10]{Rudin} has been used. Consequently,
if $\{c_k\}\in l^{\infty}$, then $$
\sup\limits_{\omega\in\mathbb B_n}\sum\limits_{k=1}^{\infty}\frac{(1-|\omega|^2)^n(1-|a_k|^2)^{n-2s}}{|1-\langle a_k,\omega\rangle|^{2n-2s}}\lesssim 1,
$$
which shows $ \mu\in\mathcal{CM}_{1-\frac{2s}{n}}.$

On the other hand, if $\mu\in\mathcal{CM}_{1-\frac{2s}{n}}$, then an application of Lemma \ref{l21} yields $\{c_k\}\in  l^{\infty}$ via
$$
|c_k|^2\lesssim\sup\limits_{\omega\in\mathbb B_n}\int_{\mathbb B_n}\frac{(1-|\omega|^2)^n}{|1-\langle z,\omega\rangle|^{2n-2s}}d\mu(z)\lesssim\|\mu\|_{\mathcal{CM}_{1-\frac{2s}{n}}}^2.
$$
\end{proof}

\begin{lemma} \cite[Proposition 1.14]{Zhu1}
\label{t45}
If $\alpha$ and $t$ are two real parameters such that neither $n+\alpha$ nor $n+\alpha+t$ is a negative integer, then 
$$
\mathsf R^{\alpha,t}\left(\frac{1}{(1-\langle z, \omega\rangle)^{n+1+\alpha}}\right)=\frac{1}{(1-\langle z,\omega\rangle)^{n+1+\alpha+t}}
\quad\forall\quad\omega\in\mathbb B_n.
$$
\end{lemma}

With the help of Lemmas \ref{t44}-\ref{t45}, we discover the following atomic decomposition.

\begin{theorem}
\label{t43}
Suppose $s\in(-\frac{1}{2},\frac{1}{2})$ and $b>n$.

{\rm(i)} If $\{a_k\}\subset\mathbb B_n$ is an $r$-lattice as described above and $\{c_k\}$ is a sequence in $\mathbb C$ such that  
$$
\sum_{k=1}^{\infty}|c_k|^2(1-|a_k|^2)^{n}\delta_{ a_k}\in\mathcal{CM}_{1-\frac{2s}{n}},
$$
then 
$$
f(z)=\sum\limits_{k=1}^{\infty}c_k\frac{(1-|a_k|^2)^{b}}{(1-\langle z, a_k\rangle)^b}\in\mathcal{HC}^s.
$$
 
{\rm (ii)} If $f\in\mathcal{HC}^s$, then there are an $r$-lattice $\{a_k\}\subset\mathbb B_n$ (as described above) and a complex sequence $\{c_k\}$ such that 
$$
f(z)=\sum\limits_{k=1}^{\infty}c_k\frac{(1-|a_k|^2)^{b}}{(1-\langle z, a_k\rangle)^b}\quad\hbox{with}\quad
\sum\limits_{k=1}^{\infty}|c_k|^2(1-|a_k|^2)^n\delta_{ a_k}\in\mathcal{CM}_{1-\frac{2s}{n}}.
$$
\end{theorem}
\begin{proof} (i) If 
$$\tilde{b}=b-s\ \ \&\ \ \tilde{c}_k=
c_k(1-|a_k|^2)^{s},
$$ 
then 
$$
\begin{cases}
\tilde{b}>n;\\
|\tilde{c}_k|^2(1-|a_k|^2)^{n-2s}=|{c}_k|^2(1-|a_k|^2)^{n}.
\end{cases}
$$
Note that 
$$
\sum\limits_{k=1}^{\infty}|c_k|^2(1-|a_k|^2)^{n-2s}\delta_{a_k}\in\mathcal{CM}_{1-\frac{2s}{n}}\Rightarrow
 \sum\limits_{k=1}^{\infty}|\tilde{c}_k|^2(1-|a_k|^2)^{n}\delta_{a_k}\in\mathcal{CM}_{1-\frac{2s}{n}}.
 $$ 
So, an application of Corollary \ref{t42} and Lemma \ref{t44}(i) derive
$$ 
f(z)=\sum\limits_{k=1}^{\infty}{c}_k\frac{(1-|a_k|^2)^{{b}}}{(1-\langle z, a_k\rangle)^{{b}}}=\sum\limits_{k=1}^{\infty}\tilde{c}_k\frac{(1-|a_k|^2)^{\tilde{b}}}{(1-\langle z, a_k\rangle)^{\tilde{b}+s}}=\mathsf{R}^{\alpha,s}\left( \sum\limits_{k=1}^{\infty}\tilde{c}_k\frac{(1-|a_k|^2)^{\tilde{b}}}{(1-\langle z, a_k\rangle)^{\tilde{b}}}\right)\in\mathcal{HC}^s.
$$

(ii) Suppose $f\in \mathcal{HC}^s$, then Theorem \ref {t41} is used to imply $\mathsf{R}^{\alpha,-s}f\in \mathcal{Q}_{1-\frac{2s}{n}}$. According to Lemma \ref{t44}(ii), there exist an $r$-lattice $\{a_k\}\subset\mathbb B_n$ and a complex sequence $\{\tilde c_k\}$ such that
$$
\mathsf{R}^{\alpha,-s}f(z)=\sum\limits_{k=1}^{\infty}\tilde{c}_k\frac{(1-|a_k|^2)^{\tilde{b}}}{(1-\langle z, a_k\rangle)^{\tilde b}},
$$
where 
$$
\tilde b>n-s\ \ \&\ \
\sum\limits_{k=1}^{\infty}|\tilde c_k|^2(1-|a_k|^2)^{n-2s}\delta_{a_k}\in\mathcal{CM}_{1-\frac{2s}{n}}.
$$
Let
$$b=\tilde{b}+s\ \ \&\ \ c_k=\tilde{c}_k(1-|a_k|^2)^{-s}.
$$ 
Taking $\mathsf {R}^{\alpha,s}$ on both sides of the last representation of $\mathsf{R}^{\alpha,s}f$, and using Lemma \ref{t45}, we get
\begin{align*}
f(z)=\sum\limits_{k=1}^{\infty}c_k\frac{(1-|a_k|^2)^{b}}{(1-\langle z, a_k\rangle)^{b}}.
\end{align*}
\end{proof}

\section{Associated Gleason problem}\label{s5}
\setcounter{equation}{0}

\subsection{Canonical example and growth}\label{s51}
For $\alpha>0$, let $\mathcal{B}_{\alpha}$ denote the $\alpha$-Bloch space of holomorphic functions $f$ in $\mathbb B_n$ with
$$
\|f\|_{\mathcal{B}_{\alpha},*}=\sup\limits_{z\in\mathbb B_n}(1-|z|^2)^{\alpha}|\nabla f(z)|<\infty.
$$
In particular, $\mathcal{B}_{1}$ is just the classical Bloch space $\mathcal B$ in $\mathbb B_n$. Further, $\mathcal{B}_{\alpha}$ becomes a Banach space equipped with the norm 
$|f(0)|+\|f\|_{\mathcal{B}_{\alpha},*}.$

\begin{example}
\label{t51}
For $s\in(-\frac{1}{2},\frac{n}{2}]$ and $a\in\mathbb B_n$, let 
$$
f_{a}(z)=\frac{(1-|a|^2)^{n-s}}{(1-\langle z, a\rangle)^{n}}.
$$
Then this function is regarded as a canonical example of $\mathcal{HC}^s$ in the sense of:
$$
\sup\limits_ {a\in\mathbb B_n} \|f_a\|_{\mathcal{HC}^s}<\infty.   
$$
\end{example}
\begin{proof}
When $s=\frac{n}{2}$, the assertion is well-known; see \cite[Theorem 4.1.7]{Zhu1}. Hence, it remains to handle the case $s\in(-\frac{1}{2},\frac{n}{2})$. It is easy to calculate
$$
|\nabla f_{a}(z)|\approx\frac{(1-|a|^2)^{n-s}\bar{a}}{|1-\langle z, a\rangle|^{n+1}}.
$$
Thus, for any $\omega\in\mathbb B_n$, Lemma \ref{l22} is utilized to derive 
\begin{align*}
&\int_{\mathbb B_n}\frac{(1-|\omega|^2)^{1+2s}}{|1-\langle z,\omega\rangle|^{n+1}}|\nabla f_{a}(z)|^2(1-|z|^2)d\nu(z)\\
&\lesssim\int_{\mathbb B_n}\frac{(1-|a|^2)^{2n-2s}(1-|\omega|^2)^{1+2s}(1-|z|^2)}{|1-\langle z,a\rangle|^{2n+2}|1-\langle z,\omega\rangle|^{n+1}}\,d\nu(z)\\
&\lesssim\frac{(1-|a|^2)^{2n-2s}(1-|\omega|^2)^{1+2s}}{(1-|a|^2)^n|1-\omega\bar{a}'|^{n+1}}\\
&=\left|\frac{1-|a|^2}{1-\omega\bar{a}'}\right|^{n-2s}\left|\frac{1-|\omega|^2}{1-\omega\bar{a}'}\right|^{1+2s}\lesssim 1
\end{align*}
Let
 $$
 \begin{cases}
 d\mu(z)=|\nabla f_{a}(z)|^2(1-|z|^2)d\nu(z);\\
p=1-\frac{2s}{n};\\
q=\frac{1+2s}{n}.
\end{cases}
$$ 
Then an application of Lemma \ref{l21} implies $\mu\in\mathcal{CM}_{1-\frac{2s}{n}}$ via
$$
\sup\limits_{\omega\in\mathbb B_n}\int_{\mathbb B_n}\frac{(1-|\omega|^2)^{nq}}{|1-\langle z,\omega\rangle|^{n(p+q)}}d\mu(z)=\sup\limits_{\omega\in\mathbb B_n}\int_{\mathbb B_n}\frac{(1-|\omega|^2)^{1+2s}}{|1-\langle z,\omega\rangle|^{n+1}}d\mu(z)\lesssim 1.
$$
According to Theorem \ref{t33}, we get 
$$
f_a\in\mathcal{HC}^s\ \ \&\ \ \sup\limits_ {a\in\mathbb B_n} \|f_a\|_{\mathcal{HC}^s}<\infty.
$$
\end{proof}

Below is the canonical growth of a function in $\mathcal{HC}^s\subseteq\mathcal{B}_{1+s}$.

\begin{lemma}
\label{t52}
Suppose $s\in(-\frac{1}{2},\frac{n}{2}]$. If $f\in\mathcal{\mathcal{HC}}^{s}$, then 
$$|\nabla f(z)|\lesssim\frac{\|f\|_{\mathcal{HC}^{s}}}{(1-|z|^2)^{1+s}}\quad\forall\quad z\in\mathbb B_n.
$$ 
Moreover, the exponent $1+s$ in the last inequality is sharp.
\end{lemma}
\begin{proof} The argument for the desired inequality is split into three cases.

{\it Case 1. } $s\in(-\frac{1}{2},0)$. Under this, the desired inequalty follows from \cite[Theorem 7.9]{Zhu1}.

{\it Case 2.} $s\in(0,\frac{n}{2})$. A special case of $p=2$ in \cite[Proposition 4.4 ]{CascanteFO} shows 
$$
|f(z)|\lesssim\frac{\|f\|_{\mathcal{HC}^s}}{(1-|z|^2)^{s}}
\quad\forall\quad z\in\mathbb B_n.
$$ 
Hence $f$ belongs to the weighted Bergman space
$$
\mathcal{L}_a^1\big(\mathbb {B}_n,\, d\nu_{1+s}(z)\big)=\left\{f\in\mathcal{H}(\mathbb B_n):\,\int_{\mathbb B_n}|f(z)|(1-|z|^2)^{s}d\nu(z)<\infty\right\}.
$$
Applying \cite[Theorem2.2]{Zhu1}, we have 
$$
f(z)=c_{s}\int_{\mathbb B_n}\frac{(1-|\omega|^{2})^sf(\omega)}{(1-\langle z, \omega\rangle)^{n+1+s}}\,d\nu(\omega)\ \ \hbox{where}\ \ 
c_{s}=\frac{\Gamma(n+s+1)}{n!\Gamma(s+1)}
$$
Therefore, an application of \cite[Proposition 1.4.10]{Rudin} and the last estimate derives
$$
|\nabla f(z)|\\
\approx \left|\int_{\mathbb B_n}\frac{\bar{\omega}(1-|\omega|^{2})^sf(\omega)}{(1-\langle z, \omega\rangle)^{n+2+s}}\,d\nu(\omega)\right|
\lesssim\frac{\|f\|_{\mathcal{HC}^{s}}}{(1-|z|^2)^{1+s}}.
$$

{\it Case 3.} $s=0$.  This follows from $\mathcal{BMOA}\subset \mathcal{B}$.

{\it Case 4.} $s=\frac{n}{2}$.  The desired inequality is just
the estimate for $\mathcal{H}^2$.

The sharpness of $1+s$ follows from an application of Example \ref{t51} to 
$$
f_{a}(z)=\frac{(1-|a|^2)^{n-s}}{(1-\langle z, a\rangle)^n}\ \ \&\ \ 
|\nabla f_{a}(a)|\approx\frac{n|a|}{(1-|a|^2)^{1+s}}\ \
\forall\ \ a\in\mathbb B_n.
$$
\end{proof}

\subsection{Solution to the associated Gleason problem}\label{s52}
Originally motivated by \cite{Gleason}, we state the general Gleason problem below: 

\medskip
{\it Let $X$ be a holomorphic function space on a domain $\Omega\subset\mathbb C^n$. Given $a=(a_1,...,a_n)\in\Omega$ and $f\in X$ with $f(a)=0$, are there $\mathsf{A}_1,...,\mathsf{A}_n\in X$ such that 
$$
f(z)-f(a)=\sum_{k=1}^{n}(z_k-a_k)\mathsf{A}_k(z)\ \ \forall\ \ z\in\Omega\ ?
$$}
\medskip
Interestingly, this problem is solvable in many holomorphic function spaces; see e.g. \cite{KerzmanN, Ortega, Zhu2, Zhang} and the relevant references therein. Even more interestingly, this problem, upon being associated with $\mathcal{HC}^s$, is still solvable.

\begin{theorem}
\label{t53}
Let $s\in(-\frac{1}{2},\frac{n}{2}]$ and $m\in\mathbb Z^+$. Then 
for each $\alpha=(\alpha_1,...,\alpha_n)$ obeying $|a|=m$ there exists a bounded linear operator $\mathsf{A}_{\alpha}$ on 
$\mathcal{HC}^{s}$ such that 
$$
f(z)=\sum_{|\alpha|=m}z^{\alpha}\mathsf{A}_{\alpha}f(z)\quad\forall\quad z\in\mathbb B_n
$$
with 
$$
\partial^\gamma f(0)=0\quad\forall\quad |\gamma|=|(\gamma_1,...,\gamma_n)|=\sum_{j=1}^n\gamma_j\in\{0,1,...,m-1\}.
$$
\end{theorem}

\begin{proof} When $s=\frac{n}{2}$, the result follows from \cite[p. 155]{Zhu1}. Thus, it is enough to handle $s\in(-\frac{1}{2},\frac{n}{2})$.
The argument is divided into two cases below.

{\it Case 1. $m=1$}. The decay of $|\nabla f|$ in Lemma \ref{t52} derives
$$
\int_{\mathbb B_n}|\nabla f(z)|(1-|z|^2)^{n}d\nu(z)\lesssim \|f\|_{\mathcal{HC}^s} \int_{\mathbb B_n}(1-|z|^2)^{n-1-s}d\nu(z)<\infty.
$$
This means 
$$
\frac{\partial f}{\partial{z_k}}\in\mathcal{L}_{a}^{1}\big(\mathbb B_n,(1-|z|^2)^nd\nu(z)\big)\quad\hbox{for}\quad k=1,2,...,n.
$$
By \cite[Theorem 2.2]{Zhu1}, we get 
$$
\frac{\partial f}{\partial{z_k}}(z)=c_{n}\int_{\mathbb B_n}\frac{\partial f}{\partial{\omega_k}}(\omega)\frac{(1-|\omega|^{2})^n}{(1-\langle z, \omega\rangle)^{2n+1}}\,d\nu(\omega)\ \ \hbox{where}\ \ 
c_n
=\frac{(2n+1)!}{(n!)^2}.
$$
For all $1\leq k\leq n$, define 
$$
\mathsf{A}_{k} f(z)=\int_{0}^{1}\frac{\partial f}{\partial{z_k}}(tz)\, dt.
$$
It is easy to see that
$$
\mathsf{A}_{k} f\in\mathcal{H}(\mathbb B_n)
$$ 
and $\mathsf{A}_{k}$ is a linear operator. Therefore, we need only to prove that $\mathsf{A}_{k}$ is bounded on $\mathcal{HC}^s$ for all $s\in(-\frac{1}{2},\frac{n}{2})$. To do so, note that
\begin{align*}
&f(z)-f(0)=\int_{0}^{1}\frac{d f}{dt}(tz)dt=\sum\limits_{k=1}^{n}\int_{0}^{1}z_k\frac{\partial f}{\partial{z_k}}(tz)dt=\sum\limits_{k=1}^{n}z_k \mathsf{A}_k f(z).
\end{align*}
So, we achieve
\begin{align*}
&\nabla\mathsf{A}_kf(z)=\nabla\int_{0}^{1}\frac{\partial f}{\partial{z_k}}(tz)\,dt\\
&=c_{n}(2n+1)\int_{0}^{1}\int_{\mathbb B_n}\Big(\frac{\partial f}{\partial{\omega_k}}(\omega)\Big)\left(\frac{\bar{\omega}(1-|\omega|^{2})^n}{(1-t\langle z, \omega\rangle)^{2n+2}}\right)\,d\nu(\omega)dt\\
&=c_{n}(2n+1)\int_{\mathbb B_n}\Big(\frac{\partial f}{\partial{\omega_k}}(\omega)\Big)\bar{\omega}(1-|\omega|^2)^n\,d\nu(\omega)\int_{0}^{1}\frac{dt}{(1-t\langle z, \omega\rangle)^{2n+2}}\\
&=c_{n}\int_{\mathbb B_n}\Big(\frac{\partial f}{\partial{\omega_k}}(\omega)\Big)\left(\frac{\bar{\omega}(1-|\omega|^2)^n}{(1-\langle z, \omega\rangle)^{2n+1}}\right)\left(\frac{1-(1-\langle z, \omega\rangle)^{2n+1}}{\langle z, \omega\rangle}\right)\,d\nu(\omega).
\end{align*}
Since 
$$
\frac{1-(1-\langle z, \omega\rangle)^{2n+1}}{\langle z, \omega\rangle}
$$
is a polynomial of $z$ and $\omega$, we have 
$$\sup\limits_{z,\,\omega\in\mathbb B_n}\left|\frac{1-(1-\langle z, \omega\rangle)^{2n+1}}{\langle z, \omega\rangle}\right|\lesssim 1.
$$
Hence 
$$
|\nabla\mathsf{A}_kf(z)|\lesssim \int_{\mathbb B_n}\frac{(1-|\omega|^2)^n|\nabla f(\omega)|}{|1-\langle z, \omega\rangle|^{2n+1}}d\nu(\omega).
$$
Note that $f\in\mathcal{HC}^s$. So, Theorem \ref{t33} is used to give that $|\nabla f(z)|^2(1-|z|^2)d\nu(z)$ belongs to $\mathcal{CM}_{1-\frac{2s}{n}}$.
This, along with Theorem \ref{l23} (by setting $a=0$ and $b=n+1$), shows that $
|\mathsf{A}_kf(z)|^2(1-|z|^2)d\nu(z)$ belongs to $\mathcal{CM}_{1-\frac{2s}{n}}$. 

 {\it Case 2. $m\geq 2$}. We proceed the argument by an induction on $m$.
For $m=1$, we have proved
$$
 f(z)=\sum\limits_{k=1}^{n}z_k \mathsf{A}_k f(z)
\quad\hbox{where}\quad
\mathsf{A}_k f\in\mathcal{HC}^s.
$$
Suppose the assertion is valid for $m-1$. Then
$$
f(z)=\sum_{|\alpha|=m-1}z^{\alpha}\mathsf{A}_{\alpha} f(z)=\sum\limits_{\alpha_1+...+\alpha_{n}=m-1}z_1^{\alpha_1}\cdots z_{n}^{\alpha_{n}}\mathsf{A}_{\alpha_1,...,\alpha_{n}} f(z),
$$
where 
$$
\mathsf{A}_{\alpha_1,...,\alpha_{n}} f\in\mathcal{HC}^s.
$$
We apply the case $m=1$ with $\mathsf{A}_{\alpha_1,...,\alpha_{n}} f$ replacing $f$ to obtain
$$
 \mathsf{A}_{\alpha_1,...,\alpha_{n}} f(z)= \sum\limits_{k=1}^{n}z_k \mathsf{A}_k \mathsf{A}_{\alpha_1,...,\alpha_{n}} f(z).
$$
Let $\mathsf{A}_{\gamma}=\mathsf{A}_{k}\mathsf{A}_{\alpha_1,...,\alpha_{n}}$. Then $\mathsf{A}_{\gamma}$ is bounded on  $\mathcal{HC}^s$. By the inductive hypothesis on $m-1$, we have 
 $$
  f(z)=\sum\limits_{k=1}^{n}\sum\limits_{\alpha_1+...+\alpha_{n}=m-1}z_kz_1^{\alpha_1}\cdots z_{n}^{\alpha_{n}} \mathsf{A}_{\gamma} f(z)=\sum_{|\beta|=m}z^{\gamma}\mathsf{A}_{\gamma} f(z).
 $$
 This proves the assertion for $m$.
\end{proof}

\section{Induced Carleson problem}\label{s6}
\setcounter{equation}{0}

\subsection{Solution to the induced Carleson problem}\label{s61}
Let $\mu$ be a nonnegative Borel measure on $\mathbb{B}_n$. Denote $\mathcal{T}_{s}^{\infty}(\mu)$ by the non-isotropic tent space of all $\mu$-measurable functions $f$ on $\mathbb{B}_n$ satisfying 
$$
\|f\|_{\mathcal{T}_{s}^{\infty}(\mu)}=\sup_{Q_r(\zeta)}\left( r^{2s-n}\int_{Q_{r}(\zeta)}|f|^2\,d\mu\right)^\frac12<\infty,
$$
where the supremum is taken over all $Q_{r}(\zeta)=\{z\in\mathbb{B}_n: \,|1-\langle z,\zeta\rangle|<r\}$.

 Referring to \cite{Xiao3} and its related references, we raise the induced Carleson problem for $\mathcal{HC}^s$ as follows:
 
 \medskip
 {\it  Suppose that $\mu$ is a nonnegative Borel measure on $\mathbb{B}_n$. What geometric nature does $\mu$ possess in order that $\mathcal{HC}^s$ embeds continuously into $\mathcal{T}_{s}^{\infty}(\mu)$\ ?}
 \medskip

Below is a solution of the preceding question.

\begin{theorem}
\label{t61}
Let $s\in(-\frac{1}{2},\frac{n}{2})$. Suppose that $\mu$ is a nonnegative Borel measure on $\mathbb{B}_n$. Then the identity operator
 $\mathsf{I}:\,\mathcal{HC}^s\mapsto\mathcal{T}_{s}^{\infty}(\mu)$ is continuous if and only if
$$
\begin{cases}
\|\mu\|_{\mathcal{CM}_{1-\frac{2s}{n}}}=\sup_{Q_r(\zeta)} \left({\frac{\mu(Q_{r}(\zeta))}{r^{n-2s}}}\right)^\frac12<\infty\ \ \hbox{for}\ \ s\in (-1/2,0);\\
 \|\mu\|_{\mathcal{LCM}_1}= \sup_{Q_r(\zeta)} \left({\frac{\mu(Q_{r}(\zeta))}{r^{n}(\log{\frac{2}{r}})^{-2}}}\right)^\frac12<\infty\ \ \hbox{for}\ \ s=0;\\
  \|\mu\|_{\mathcal{CM}_1}=\sup_{Q_r(\zeta)}\left({\frac{\mu(Q_{r}(\zeta))}{r^{n}}}\right)^\frac12<\infty\ \ \hbox{for}\ \ s\in(0,n/2),
  \end{cases}
  $$
where $``sup_{Q_r(\zeta)}"$ means the supremum ranging over all $Q_{r}(\zeta)=\{z\in\mathbb{B}_n: \,|1-\langle z,\zeta\rangle|<r\}$.

\end{theorem}
\begin{proof} Three cases are considered below.

{\it Case 1. $s\in (-\frac{1}{2},0)$}. Under this situation, one has $\mathcal{HC}^s=\mathcal{B}_{1+s}$. If $\mathsf{I}$ is bounded, then taking $f(z)=1$ we get $\|\mu\|_{\mathcal{CM}_{1-\frac{2s}{n}}}<\infty.$ Conversely, if $\|\mu\|_{\mathcal{CM}_{1-\frac{2s}{n}}}<\infty$, then 
$$
\|f\|_{\mathcal{T}^\infty_s(\mu)}\lesssim \|f\|_{\mathcal{HC}^s}\|\mu\|_{\mathcal{CM}_{1-\frac{2s}{n}}}.
$$

{\it Case 2. $s=0$}. Under this case, one has $\mathcal{HC}^s=\mathcal{BMOA}$. The desired assertion has been verified in \cite{PengOu2}.

{\it Case 3. $s\in (0,\frac{n}{2})$}. If $\mathsf{I}$ is continuous, then for any $f\in\mathcal{HC}^s$ one has $f\in\mathcal{T}^{\infty}_s(\mu).$ Now, for any $\omega\in\mathbb{B}_n$ let 
$$
f_{\omega}(z)=\frac{(1-|\omega|^2)^{n-s}}{(1-\langle z, \omega\rangle)^{n}}.
$$
Example \ref{t51} is used to imply that $f_{\omega}\in\mathcal{HC}^s$. Given a Carleson tube ${Q}_r(\zeta)$, let $\omega=(1-r)\zeta$. Then 
$$
r^{-n}{\mu(Q_{r}(\zeta))}\lesssim r^{2s-n}\int_{{Q}_r(\zeta)}|f_{\omega}(z)|^2
d\mu(z)\lesssim\|f_{\omega}\|^2_{\mathcal{HC}^s}\lesssim 1.
$$
Accordingly, $\|\mu\|_{\mathcal{CM}_1}<\infty.$

On the other hand, if $\|\mu\|_{\mathcal{CM}_1}<\infty.$ Then we need to prove that $\mathsf{I}$ is continuous. Let $f\in\mathcal{HC}^s$. For $\zeta\in\mathbb{S}_n$ and $0<r<1$, write 
$$
a=(1-r)\zeta\,\, \&\ \ Q_r({\zeta})=\{z\in\mathbb{B}_n:|1-\langle z,\zeta\rangle|<r\}.
$$
By the H\"older inequality, we obtain
$$ 
{r^{2s-n}}\int_{Q_r({\zeta})}|f(z)|^2d\mu(z)
\lesssim
{r^{2s-n}}\int_{Q_r({\zeta})}|f(a)|^2d\mu(z)+{r^{2s-n}}\int_{Q_r({\zeta})}|f(z)-f(a)|^2d\mu(z)\equiv J_1+J_2.
$$
For $J_1$, the decay of $f$ found in Lemma \ref{t52} gives 
$$
|f(a)|\lesssim{(1-|a|^2)^{-s}}{\|f\|_{\mathcal{HC}^s}}\approx r^{-s}{\|f\|_{\mathcal{HC}^s}}.
$$
Thus
$$
J_1\lesssim{r^{-n}\|f\|_{\mathcal{HC}^s}^2}\int_{Q_r({\zeta})}d\mu(z)\lesssim \|\mu\|^2_{\mathcal{CM}_1}\|f\|_{\mathcal{HC}^s}^2.
$$
For $J_2$, both $\|\mu\|_{\mathcal{CM}_1}<\infty$ and \cite{Hormander} are utilized to derive
$$
\int_{\mathbb{B}_n}|g(z)|^2d\mu(z)
\lesssim\|\mu\|^2_{\mathcal{CM}_1}\int_{\mathbb{S}_n}|g(\zeta)|^2 d\sigma(\zeta)\ \ \forall\ \  g\in\mathcal{H}^2.
$$
This last estimate, along with 
$$
|1-\langle a, z\rangle|=|(1-r)(1-\langle \zeta, z\rangle)+r|<2r\ \ \forall\ \ a=(1-r)\zeta\ \&\ z\in Q_r(\zeta),
$$
deduces
\begin{align*}
&J_2\lesssim{(1-|a|^2)^{2s}}\int_{Q_r({\zeta})}|f(z)-f(a)|^2\frac{(1-|a|^2)^n}{|1-\langle z, a\rangle|^{2n}}\,d\mu(z)\\
&\lesssim{(1-|a|^2)^{2s}}\int_{\mathbb{B}_n}|f(z)-f(a)|^2\frac{(1-|a|^2)^n}{|1-\langle z, a\rangle|^{2n}}\,d\mu(z)
\\
&\lesssim\|\mu\|^2_{\mathcal{CM}_1} (1-|a|^2)^{2s}\int_{\mathbb{S}_n}|f(\zeta)-f(a)|^2\frac{(1-|a|^2)^n}{|1-\langle \zeta, a\rangle|^{2n}}d\sigma(\zeta)
\\
&=\|\mu\|^2_{\mathcal{CM}_1}(1-|a|^2)^{2s}\int_{\mathbb{S}_n}|f\circ\varphi_{a}(\zeta)-f(a)|^2d\sigma(\zeta)
\\
&=\|\mu\|^2_{\mathcal{CM}_1}(1-|a|^2)^{2s}\|f\circ\varphi_{a}-f(a)\|_{\mathcal{H}^2}^2
\\
&\lesssim\|\mu\|^2_{\mathcal{CM}_1} \|f\|_{\mathcal{HC}^s}^2.
\end{align*}
From the estimates for $J_1$ and $J_2$, we know
$$
{r^{2s-n}}\int_{Q_r({\zeta})}|f(z)|^2d\mu(z)\lesssim \|\mu\|^2_{\mathcal{CM}_1}\|f\|_{\mathcal{HC}^s}^2<\infty,
$$
thereby finding
$$
\|f\|_{\mathcal{T}_{s}^{\infty}(\mu)}\lesssim\|\mu\|_{\mathcal{CM}_1}
\|f\|_{\mathcal{HC}^s},
$$
i.e.,  the identity operator $\mathsf{I}: \mathcal{HC}^s\mapsto\mathcal{T}^\infty_s(\mu)$ is continuous.
\end{proof}

\subsection{Riemann-Stieltjes operator}\label{s62} Denote by $\mathsf{T}_g$ and $\mathsf{S}_g$ the so-called Riemann-Stieltjes operator and the associate Riemann-Stieltjes operator with a given symbol $g\in\mathcal{H}(\mathbb B_n)$, respectively: 
 $$
 \mathsf{T}_g f(z)=\int_{0}^{1}f(tz)\mathsf{R}g(tz)\frac{dt}{t}\ \ \&\ \
\mathsf {S}_g f(z)=\int_{0}^{1}g(tz)\mathsf{R}f(tz)\frac{dt}{t}
\quad\forall\quad f\in\mathcal{H}(\mathbb B_n)\ \ \&\ \ z\in\mathbb B_n,
$$
 where 
 $$
 \mathsf{R}f(z)=\sum\limits_{k=1}^n z_k\frac{\partial f}{\partial z_k}(z)
 $$ 
 is still the radial derivative of $f$. It is easy to examine the following formulas:
 $$
 \begin{cases}
 \mathsf{T}_g(f)=\mathsf{S}_f(g);\\
 \mathsf{M}_gf(z)=f(z)g(z)=f(0)g(0)+\mathsf {T}_g f(z)+\mathsf {S}_g f(z);\\
 \mathsf{R}\circ\mathsf {T}_g f(z)=f(z)\mathsf{R}g(z).
 \end{cases}
 $$
 
 The Riemann-Stieltjes operators on different holomorphic function spaces have been investigated in many papers including \cite{AlemanC, AlemanS, AhernS, GirelaP, Hu1, Hu3, PengOu2, Pommerenke, Siskakis, SiskakisZ, Xiao4, Xiao07, Xiao3}. Now, as an interesting by-product of Theorem \ref{t61}, we can discover the continuity of a Riemann-Stieltjes operator acting on $\mathcal{HC}^s$. 
  
 \begin{theorem}
\label{t62} Let $s\in(-\frac{1}{2},\frac{n}{2})$, $g\in\mathcal{H}(\mathbb{B}_n)$ and $d\mu_g(z)=(1-|z|^2)|\mathsf{R}f(z)|^2d\nu(z)$. Then:

 {\rm(i)} $\mathsf{T}_g$ is continuous on  $\mathcal{HC}^s$ if and only if 
$$
\begin{cases}
\|\mu_{g}\|_{\mathcal{CM}_{1-\frac{2s}{n}}}<\infty\quad\hbox{as}\quad s\in (-\frac{1}{2},0);\\
\|\mu_{g}\|_{\mathcal{LCM}_1}<\infty\quad\hbox{as}\quad s=0;\\
\|\mu_{g}\|_{\mathcal{CM}_1}<\infty\quad\hbox{as}\quad s\in (0,\frac{n}{2}).
\end{cases}
$$

 {\rm(ii)} $\mathsf{S}_g$ is continuous on $\mathcal{HC}^s$ if and only if $\|g\|_{\mathcal{H}^{\infty}}=\sup_{z\in\mathbb B_n}|f(z)|<\infty$.
 \end{theorem}

\begin{proof} (i) The equivalence under $s\in (-\frac{1}{2},0]$ follows from the structure table of $\mathcal{HC}^s$ and \cite{OrtegaF, Xiao4}. And, the equivalence under $s\in (0, \frac{n}{2})$ follows from using both
$\mathsf{R}\circ\mathsf {T}_g f=f\mathsf{R}g$ and Theorem \ref{t61}.

(ii) If $\|g\|_{\mathcal{H}^{\infty}}<\infty$, then for any Carleson tube $Q_r(\zeta)\subseteq\mathbb S_n$ we have
\begin{align*}
&r^{2s-n}\int_{Q_r(\zeta)}|\mathsf{R}\circ\mathsf{S}_{g}f(z)|^2(1-|z|^2)d\nu(z)\\&=r^{2s-n}\int_{Q_r(\zeta)}|\mathsf{R}\circ\mathsf{T}_{f}g(z)|^2(1-|z|^2)d\nu(z)\\&=r^{2s-n}\int_{Q_r(\zeta)}|g(z)|^2|\mathsf{R} f(z)|^2(1-|z|^2)d\nu(z)\\&
\lesssim \|g\|_{\mathcal{H}^{\infty}}^2\|f\|_{\mathcal{HC}^s}^2,
\end{align*}
whence reaching
$$
\|\mathsf{S}_gf\|_{\mathcal{HC}^s}\lesssim \|g\|_{\mathcal{H}^{\infty}}\|f\|_{\mathcal{HC}^s},
$$
which in turn implies that $\mathsf{S}_g$ is continuous on $\mathcal{HC}^s$.

Conversely, suppose $\mathsf{S}_g$ is continuous on $\mathcal{HC}^s$. Then its operator norm $\|\mathsf{S}_g\|$ is finite. Given any $\omega$ with $\frac{2}{3}<|\omega|<1$. Then, for any $\zeta\in\mathbb{S}_n$ there exists $0<r<1$ such that the corresponding Carleson tube $Q_r(\zeta)$ enjoys
$$ 
E(\omega,{1}/{2})=\Big\{z\in\mathbb{B}_n:\ |\varphi_{\omega}(z)|<{1}/{2}\Big\}\subset Q_r(\zeta)\ \ \&\ \  1-|\omega|^2\approx r.
$$ 
When $z\in E(\omega,{1}/{2}) $, some calculations show 
$$
\begin{cases}
\nu(E(\omega,{1}/{2}))\approx (1-|\omega|^2)^{n+1};\\
|\langle z, \omega\rangle|\gtrsim 1;\\
1-|\omega|^2\approx 1-|z|^2\approx |1-\langle z,\omega\rangle|\approx r.
\end{cases}
 $$
Choosing the test function 
$$
f_{\omega}(z)=\frac{(1-|\omega|^2)^{n-s}}{(1-\langle z, \omega\rangle)^{n}},
$$ 
and using Example \ref{t51}, we gain
$$
\sup\limits_ {\omega\in\mathbb B_n} \|f_{\omega}\|_{\mathcal{HC}^s}<\infty.
$$ 
Since $|g|^2$ is subharmonic in $\mathbb B_n$, one has 
\begin{align*}
&|g(\omega)|^2\lesssim{(1-|\omega|^2)^{-n-1}} \int\limits_{E(\omega,{1}/{2})}|g(z)|^2d\nu(z)\\
&\lesssim{r^{2s-n}} \int\limits_{E(\omega,{1}/{2})}\frac{|g(z)|^2(1-|z|^2)^{2n-2s+1}}{|1-\langle z,\omega\rangle|^{2(n+1)}}d\nu(z)\\
&\lesssim{r^{2s-n}} \int\limits_{E(\omega,{1}/{2})}\frac{|g(z)|^2|\langle z,\omega\rangle|^2(1-|z|^2)^{2n-2s+1}}{|1-\langle z,\omega\rangle|^{2(n+1)}}d\nu(z)\\
&\lesssim{r^{2s-n}}\int\limits_{Q_r(\zeta)}|g(z)|^2|\mathsf{R}f_{\omega}(z)|^2(1-|z|^2)d\nu(z)\\&\lesssim \|\mathsf{S}_g(f_{\omega})\|_{\mathcal{HC}^s}^2\lesssim\|\mathsf{S}_g\|^2.
\end{align*}
This, together with the maximum modulus principle, ensures $\|g\|_{\mathcal{H}^{\infty}}<\infty$.
\end{proof}

In light of Theorem \ref{t62}, the pointwise multipliers of $\mathcal{HC}^s$ can be readily determined as a unification of the well-known results in \cite{CascanteFO, OrtegaF, ZhuR}

\begin{corollary}
\label{t63}
Let $s\in(-\frac{1}{2},\frac{n}{2})$, $g\in \mathcal{H}(\mathbb{B}_n)$ and $d\mu(z)_{g}=(1-|z|^2)|\mathsf{R}f(z)|^2d\nu(z)$. Then  $\mathsf{M}_g$ is continuous on $\mathcal{HC}^s$ if and only if 
$$
\begin{cases}
g\in\mathcal{HC}^s\quad\hbox{as}\quad s\in (-\frac{1}{2},0);\\
\|g\|_{\mathcal{H}^{\infty}}+\|\mu_{g}\|_{\mathcal{LCM}_1}<\infty\quad\hbox{as}\quad s=0;\\
\|g\|_{\mathcal{H}^{\infty}}<\infty\quad\hbox{as}\quad s\in (0,\frac{n}{2}).
\end{cases}
$$
\end{corollary}

\end{document}